\documentclass[journal]{IEEEtran}

\usepackage{afterpage}
\usepackage{amsmath}  
\usepackage{amssymb}  
\usepackage{amsthm}
\usepackage{algorithm}
\usepackage{algorithmic}
\usepackage{bbm}
\usepackage{caption}
\usepackage{subcaption}
\usepackage{cite}
\usepackage{float}
\usepackage{graphicx}
\usepackage{epstopdf}
\usepackage[colorlinks]{hyperref}
\usepackage{psfrag}
\usepackage{setspace}

\newtheorem{theorem}{Theorem} 
\newtheorem{lemma}{Lemma}
\newtheorem{asm}{Assumption}
\DeclareMathOperator*{\argmin}{argmin}

\title{\LARGE \bf
Multi-agent Planning for thermalling gliders using multi level graph-search}

\author{
Muhammad Aneeq uz Zaman\thanks{Graduate Student, Dept. of Mechanical Science and Engineering, University of Illinois Urbana-Champaign, mazaman2@illinois.edu.}~,
Aamer Iqbal Bhatti \thanks{Professor of Controls and Signal Processing, Department of Electrical Engineering, aamer987@gmail.com.}\\
}

\begin{document}

\maketitle
\thispagestyle{empty}
\pagestyle{empty}

\begin{abstract}
This paper solves a path planning problem for a group of gliders. The gliders are tasked with visiting a set of \emph{interest points}. The gliders have limited range but are able to increase their range by visiting special points called \emph{thermals}. The problem addressed in this paper is of path planning for the gliders such that, the total number of interest points visited by the gliders is maximized. This is referred to as the multi-agent problem. The problem is solved by first decomposing it into several single-agent problems. In a single-agent problem a set of interest points are allocated to a single glider. This problem is solved by planning a path which maximizes the number of visited interest points from the allocated set. This is achieved through a uniform cost graph search, as shown in our earlier work. The multi-agent problem now consists of determining the best allocation (of interest points) for each glider. Two ways are presented of solving this problem, a brute force search approach as shown in earlier work and a Branch\&Bound type graph search. The Branch\&Bound approach is the main contribution of the paper. This approach is proven to be optimal and shown to be faster than the brute force search using simulations. 
\end{abstract}

\section{INTRODUCTION} \label{sec:Introduction}
UAVs are taking over a wide variety of applications in the modern world. These applications include precision agriculture \cite{zhang2012application}, public safety \cite{li2015drone} and surveillance \cite{kingston2008decentralized}. UAVs are effective particularly in long duration missions like surveillance applications. This is due to the increased reliability of the UAVs over their manned counterparts. But UAVs have limited range due to their steadily decreasing fuel. This constraint hampers the UAV's ability to carry out its mission. To increase the range of UAVs alternate methods of refueling need to be investigated. In this paper we will look at a specific environmental phenomenon which can be utilized to refuel the UAVs in-flight and hence increase its range.

The phenomenon which can be used to 'refuel' the UAVs is called a thermal. Thermals are columns of rising hot air which can help glider-like UAVs gain height and hence increase their range. The use of thermals has been shown to be a viable method of increasing flight range of gliders \cite{c1}. Therefore, we will be looking at how best to use this environmental phenomena to increase the range of the UAVs and potentially enable it to perform its mission better.

In this paper we deal with a particular surveillance mission where a set of glider-like UAVs (henceforth called gliders) have to visit a set of interest points. As the gliders travel they are constantly losing height and so their range is limited. We also have a set of thermals, which the gliders can visit to gain height and increase their range. Hence thermals can potentially enable the gliders to visit even more interest points. 

This problem is referred to as the multi-agent problem in this paper. To solve this problem we need to plan paths for the gliders which are \emph{feasible} and \emph{valid}. A path is considered feasible if it satisfies the dynamic constraints of the glider and valid if the glider's height remains non-negative as it follows the path. Furthermore, the paths should maximize the number of interest points visited by the glider. Finally, if there are multiple (valid and feasible) paths which visit the same number of interest points, the one with the least arclength needs to be chosen.


We first introduce a class of curves which we refer to as composite paths. A composite path provides a feasible path for a glider over a sequence of waypoints. Then the multi-agent problem is solved by using the composite paths and a multi level graph search. This approach is an improvement over our earlier work \cite{uzpath}. The multi level graph search is composed of upper and lower level graph search. 

The multi-agent problem is first decomposed into multiple single-agent problems. In a single-agent problem, a subset of interest points is \textit{allocated} to a glider. A valid composite path is determined for the glider, which maximizes the number of 'allocated' visited interest points. Note that a glider may not be able to visit all of its allocated interest points due to validity constraints. As shown in our earlier work \cite{uzpath} we use a uniform cost graph search to obtain the solution. This is called lower level graph search because it forms the lower level of the multi level graph search.

The multi-agent problem is now solved by finding the best allocation of interest points to each glider. The allocation which maximizes the number of interest points visited by the gliders collectively, is the best allocation. In our earlier work \cite{uzpath} this problem was solved via a brute force search method. This paper introduces a Branch\&Bound graph search type method to solve this problem. This is called upper level graph search since it forms the upper level of the multi level graph search. We introduce two new notions of \emph{ideality} and \emph{weakness} which help in proving the optimality of the Branch\&Bound algorithm. We also show that the Branch-and-Bound approach is faster than the brute force approach. The main contribution of the paper is the Branch\&Bound algorithm and its proof of optimality. 

\subsection{Related Work}
To the best of the author's knowledge the multi-agent problem has not been dealt with in literature. There are two main aspects of the problem. Firstly, planning a feasible path for a glider given a waypoint visitation order. Secondly, the determining the best waypoint visitation order for each glider.

Planning feasible paths for the gliders entails finding a class of interpolating curves whose curvature and sharpness profiles are bounded. These bounds depend on the dynamics of the gliders as explained in section \ref{sec:probstat}. Dubin's curves are widely used to plan paths with bounded curvature, like in \cite{barraquand1989nonholonomic} and \cite{laumond1994motion}. Dubin's curves have bounded but discontinuous curvature which leads to unbounded sharpness. Paths with continuous curvature have been proposed by authors of \cite{boissonnat1994note} and \cite{kostov1998irregularity}, but these papers do not consider bounds on sharpness. Papers like \cite{c29} proposed iterative methods for finding curves with bounds on curvature and sharpness. Authors of paper \cite{scheuer1997continuous} propose a closed form but suboptimal method for planing paths with bounded curvature and sharpness. 

In our earlier work \cite{uzpath} we modified this approach to plan feasible paths for the gliders. The Continuous Curvature (CC) turn (as introduced in \cite{scheuer1997continuous}) is used to plan a path from a given starting point and orientation to an end point. The curve has a piece-wise linear curvature profile which satisfies the feasibility constraints by design. Moreover the arclength of the resulting curve can also be calculated in closed form. This is helpful in determining the validity of the path.

The problem of determining the best waypoint visitation order for each glider is a variant of the Team Orienteering Problem (TOP). It can be thought of as an asymmetric TOP with refueling points. In TOP \cite{c14} a group of robots have to visit a set of points while minimizing the total distance traveled. Exact approaches to solve the (Symmetric) TOP have been proposed in \cite{c8} and \cite{c9}. There has also been work in heuristic methods to solve this problem, like in \cite{c19}, \cite{c10}. Authors of \cite{c13} proposes a multi level graph-search based method for this purpose. 

This paper uses the idea of multi level graph search to solve the multi-agent problem. The reason for using this approach over others is that multi level graph search is very effective at dealing with highly nonlinear constraints like the validity constraint and it is provably optimal. The lower level of the graph search, which is a uniform cost graph search, determines the best waypoint visitation order for a glider given that a set of waypoints have been allocated to it. This algorithm was developed in our earlier work \cite{uzpath}. The upper level of the graph search, which is a Branch\&Bound graph search, determines the best allocation of waypoints to gliders. This algorithm is proved to be optimal. It is also shown to be faster than the brute force search approach used in \cite{uzpath}. This algorithm and its proof of optimality are considered to be the main contributions of this paper.


The paper is organized as follows. Section \ref{sec:probstat} formulates the problem rigorously. Section \ref{sec:point2point} deals with, planning a composite path for a glider over any given waypoint visitation order. Section \ref{sec:singleglider} solves the single-agent problem using the lower-level graph search and introduces the new concepts of \emph{ideality} and \emph{weakness}. These concepts are used in Section \ref{sec:multiglider} which describes how to solve the multi-agent problem using the upper-level graph search. In sections \ref{sec:singleglider} and \ref{sec:multiglider} the notation related to the graph search is developed separately. This notation is specific to the solution presented in the paper while the notation in section \ref{sec:probstat} is general. Section \ref{sec:simul} presents the simulation results of the proposed solution. Section \ref{sec:conc} concludes the paper. Notation pertinent to each section is developed in the section itself.

\section{Problem Statement} \label{sec:probstat}
This section will present the mathematical formulation for the single-agent and multi-agent problems separately. First, however,  we introduce some details of the problem scenario.

There are a total of $n_{ip}$ \emph{interest points} each denoted by $ip_j$ where $j \in \{1,..,n_{ip}\}$. These points are located at positions $p_{ip_j} \in \mathbb{R}^2$. There are $n_t$ \emph{thermals} and each is denoted by $t_k$ where $k \in \{1,..,n_t\}$. The height gained by visiting thermals and the positions of the thermals are assumed to be known a priori. Moreover, the thermals are treated as points and we assume that the gliders gain height instantaneously when they visit a thermal. They are positioned at $p_{t_k} \in \mathbb{R}^2$ and the height any glider can gain by visiting them is denoted by $h_{t_k} \in \mathbb{R}^+$. Similarly we define that the height attained by visiting an interest point is $0$, $h_{ip_j}=0, \forall j$.

We have a total of $n_g$ identical gliders. Each glider $i$ has a pre-specified start position $p^0_{i} \in \mathbb{R}^2$ and orientation $\theta^0_{i} \in [-\pi,\pi]$. The gliders must reach their respective final positions located at $p^f_{i} \in \mathbb{R}^2$. The starting height of the gliders is denoted by $h^0_{i} \in \mathbb{R}^+$. Figure \ref{fig:figdispvisorder} presents a simple scenario with two gliders with starting points ($p^0_1$, $p^0_2$) as circles and final positions ($p^f_1$, $p^f_2$) as crosses. There is also a thermal $t_1$ as a diamond and three interest points $ip_1$, $ip_2$ and $ip_3$ as squares.

Ideally, the height lost per horizontal distance traveled by the glider should be minimized. This corresponds to minimizing the angle of descent $\gamma_{d,i}$ of the glider. As shown in \cite{c28} this can be achieved by choosing an appropriate value of angle of attack. For this angle of attack, the corresponding angle of descent is $\gamma_{d_{min}}$. Since the gliders are identical, it is the same for all gliders. We assume that a controller ensures $\gamma_{d,i} = \gamma_{d_{min}}$ for all gliders. Now we may ignore the vertical degree of freedom of the gliders.

\subsection{Single-agent problem}
Each glider $i$ is allocated a set of interest points denoted by $\boldsymbol{\xi}_i$, where $\boldsymbol{\xi}_i \subseteq \{ip_{1},..,ip_{n_{ip}}\}$. The path glider $i$ takes through $\boldsymbol{\xi}_i$ is called a \emph{composite path}. The composite path is composed of multiple legs, where each leg is a path from one waypoint to another.

\begin{figure}[h!] 
	\centering
	\includegraphics[width=0.4 \textwidth]{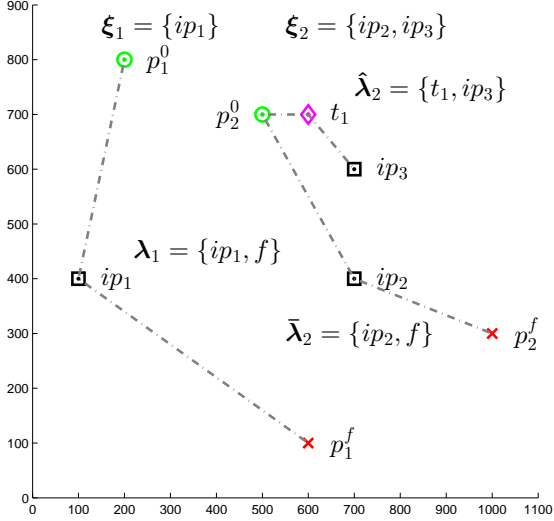}
	\caption{Some possible waypoint visitation orders.}
	\label{fig:figdispvisorder}
\end{figure} 

\subsubsection{Visitation order}The order in which the glider $i$ visits the waypoints is denoted by $\boldsymbol{\lambda}_i$ such that,
\begin{multline*}
\boldsymbol{\lambda}_i = \{\lambda_{i,j}\} : \lambda_{i,j} \in \boldsymbol{\xi}_i \cup \{t_1,..,t_{n_t}\} \cup \{f\},\\ i \in \{1,..,n_g\}, j \in \{1,..,n(\boldsymbol{\lambda}_i)\} \,,
\end{multline*}
where $f$ represents the final position of the glider and $n()$ is the cardinality operator. $\lambda_{i,j}$ is the $j$th waypoint in $i$th glider's composite path. The definition given above states that it can only be an interest point allocated to $i$, a thermal or the final position of $i$. Notice that a glider does not have to visit all its allocated interest points. The total number of unvisited interest points in $\boldsymbol{\lambda}_i$ is represented by $K_L(\boldsymbol{\lambda}_i,\boldsymbol{\xi}_i)$,
\begin{equation} \label{eq:unvisintpt}
K_L(\boldsymbol{\lambda}_i,\boldsymbol{\xi}_i) =  n(\boldsymbol{\xi}_i) - \sum^{n_{ip}}_{k=1} \mathbbm{1}_{\boldsymbol{\lambda}_i}(k)
\end{equation}
The function $\mathbbm{1}_{\boldsymbol{\lambda}_i}(k)$ is an indicator function which is $1$ if $k \in \boldsymbol{\lambda}_i$ and 0 otherwise.

Figure \ref{fig:figdispvisorder} shows a particular scenario in which the first interest point $ip_1$ has been alloted to glider 1 and $ip_2$ and $ip_3$ have been alloted to glider 2. $\hat{\boldsymbol{\lambda}}_2 = \{t_1,ip_3\}$ and $\bar{\boldsymbol{\lambda}}_2 = \{ip_2,f\}$ are two possible visitation orders for glider 2 and $\boldsymbol{\lambda}_1 = \{ip_1,f\}$ is a possible visitation order for glider 1. In the paper sometimes we use accents like $\hat{}$, $\bar{ }$ and $'$ to represent specific visitation orders and allocations. From the figure we can see that $K_L(\bar{\boldsymbol{\lambda}}_2,\boldsymbol{\xi}_2) = K_L(\hat{\boldsymbol{\lambda}}_2,\boldsymbol{\xi}_2) = 1$ and $K_L(\boldsymbol{\lambda}_1,\boldsymbol{\xi}_1) = 0$.

\subsubsection{Composite path and height profile}
Each $\boldsymbol{\lambda}_i$ has a composite path associated with it which visits all the waypoints in $\boldsymbol{\lambda}_i$. The $j$th leg in the composite path is denoted by $r_{i,j}$. It is parameterized by arclength of the leg $l_{i,j} \in [0,l^f_{i,j}]$, where $l^f_{i,j}$ is the total arclength of the leg. The glider dynamics have non-holonomic constraints, meaning $r_{i,j}$ is defined by the orientation of the glider $\theta_{i,j}$,
\begin{equation}
r_{i,j}(l_{i,j},\boldsymbol{\lambda}_i) =  \int\limits_{0}^{l_{i,j}}
\left(
\begin{array}{c} \cos(\theta_{i,j}(u)) \\ \sin(\theta_{i,j}(u)) \end{array}
\right) du + r_{i,j}(0,\boldsymbol{\lambda}_i) \,,
\end{equation}
\begin{equation}
	\kappa_{i,j}(l_{i,j},\boldsymbol{\lambda}_i) = d\theta_{i,j}(l_{i,j})/dl_{i,j} \,,
\end{equation}
\begin{equation}
	\sigma_{i,j}(l_{i,j},\boldsymbol{\lambda}_i) = d\kappa_{i,j}(l_{i,j})/dl_{i,j} \,,
\end{equation}
\begin{equation} \label{eq:dh/dl}
	d h_{i,j}(l_{i,j},\boldsymbol{\lambda}_i)/d l_{i,j} = -\tan(\gamma_{d_{min}}) \,,
\end{equation}
$\kappa_{i,j}$ and $\sigma_{i,j}$ denote the curvature and sharpness of the leg, respectively. $h_{i,j}$ represents the height of the glider, which is decreasing at a constant rate throughout a leg. $\boldsymbol{\lambda}_i$ in the above stated equations may sometimes be omitted for simplicity. The first leg must satisfy start configuration constraints,
\begin{equation} \label{boundarycond}
r_{i,1}(0) = p^0_{i},\hspace{0.2cm} \theta_{i,1}(0) = \theta^0_{i}, \hspace{0.2cm} \forall i \,.
\end{equation}
All the following legs must start at corresponding waypoints in $\boldsymbol{\lambda}_i$ and maintain continuity of orientation with the last leg,
\begin{equation}
r_{i,j}(0) = p_{\lambda_{i,j-1}}, \hspace{0.2cm} \theta_{i,j}(0) = \theta_{i,j-1}(l^f_{i,j}) \,.
\end{equation}
Similarly, each leg $j$ in the composite path should end at either its corresponding waypoint or the final position,
\begin{equation} \label{boundarycond2}
r_{i,j}(l^f_{i,j}) = \left\{ \begin{array}{ll}
p_{\lambda_{i,j}} & \text{if } \lambda_{i,j} \neq f\\
p^f_{i} &  \text{if } \lambda_{i,j} = f \,.
\end{array} \right.
\end{equation}
Each leg should also satisfy the continuity of curvature between consecutive legs,
\begin{equation} \label{contposorientcurv}
\kappa_{i,j+1}(0) = \kappa_{i,j}(l^f_{i,j})\,.
\end{equation}
The glider has a starting height $h^0_{i}$ and whenever the glider visits a thermal it gets an increase in its height,
\begin{equation*}
h_{i,1}(0) = h^0_{i}, \hspace{0.2cm} h_{i,j+1}(0) = h_{i,j}(l^f_{i,j}) + h_{\lambda_{i,j}} \,.
\end{equation*}

\subsubsection{Feasibility and Validity}The composite paths must satisfy the dynamic constraints on the glider. The gliders have a constraint on the maximum roll angle, which means that the glider cannot execute tight turns. This corresponds to a maximum curvature constraint on the path. Moreover, there is an upper limit on the roll rate the glider can achieve. This corresponds to an upper limit on the sharpness of the path. Hence, a path must have bounded curvature and sharpness,
\begin{equation} \label{cons:kappa_max}
| \kappa_{i,j}(l_{i,j}) | \leq \kappa_{max}: 0 \leq l_{i,j} \leq l^f_{i,j} \,, \forall i, \forall j \,,
\end{equation}
\begin{equation} \label{cons:sigma_max}
| \sigma_{i,j}(l_{i,j}) | \leq \sigma_{max}: 0 \leq l_{i,j} \leq l^f_{i,j} \,, \forall i, \forall j \,.
\end{equation}
In the paper we assume that $\kappa_{max}$ and $\sigma_{max}$ have been predetermined for the gliders. This is called the \emph{feasibility} constraint. 

Furthermore, the height of the glider must always be positive. This is called the \emph{validity} constraint,
\begin{equation*}
h_{i,j}(l_{i,j}) \geq 0 \text{ for } 0 \leq l_{i,j} \leq l^f_{i,j}\,, \forall i, \forall j \,.
\end{equation*}
This condition is equivalent to $h_{i,j} (l^f_{i,j}) > 0$, since $h_{i,j}$ is constantly decreasing throughout a leg as per equation \eqref{eq:dh/dl}. This constraint can be expressed as a constraint on the arclength of the composite path as shown below.
\begin{equation} \label{cons:h}
h_{i,j} (l^f_{i,j}) = h^0_{i} + \sum\limits_{l = 1}^{n_t} \mathbbm{1}_{\boldsymbol{\lambda}_i}(t_l)h_{t_l} - \tan(\gamma_{d_{min}}) \sum_{k=1}^{j} l_{i,k}^f > 0 \,.
\end{equation}
this is due to the fact that $h_{i,j} (l^f_{i,j})$ is the height of the glider at the start of the composite path minus the net height lost while traveling.

\subsubsection{Optimization Problem}The problem is to find a visitation order which minimizes unvisited allocated interest points. The order should also be complete, meaning that it ends at the final position for that particular glider. If there are multiple visitation orders which satisfy these criteria, we should find the one with smallest arclength of the composite path,
\begin{multline} \label{objfuncsingle}
\argmin_{\boldsymbol{\lambda}^*_i} \sum_{\forall j} l^f_{i,j}
\hspace{0.2cm} \text{s.t.} \hspace{0.2cm}
\boldsymbol{\lambda}^*_i \in \argmin_{\boldsymbol{\lambda}_i}  K_L(\boldsymbol{\lambda}_i,\boldsymbol{\xi}_i) \\
\text{s.t. } \lambda_{i,n(\boldsymbol{\lambda}_i)} = f \text{ and } r_{i,j}(\boldsymbol{\lambda}_i) \text{ satisfies } \eqref{cons:kappa_max}-\eqref{cons:h} \,.
\end{multline}

The order $\boldsymbol{\lambda}_i$ which satisfies \eqref{objfuncsingle}, represents the best visitation order for a given waypoint allocation $\boldsymbol{\xi}_i$ to glider $i$. We call this the \emph{optimal} visitation order $\boldsymbol{\lambda}^{\boldsymbol{\xi}_i}_i$ for waypoint allocation $\boldsymbol{\xi}_i$. Similarly, the optimal path and the number of unvisited interest points for allocation $\boldsymbol{\xi}_i$ are denoted similarly as $r^{\boldsymbol{\xi}_i}_{i,j}$ and $K_U(\boldsymbol{\xi}_i)$, respectively.
\begin{equation} \label{lambdaxi=lambdaphi}
	\boldsymbol{\lambda}^{\boldsymbol{\xi}_i}_i = \boldsymbol{\lambda}_i, \hspace{0.2cm} r^{\boldsymbol{\xi}_i}_{i,j} = r_{i,j}(\boldsymbol{\lambda}_i), \hspace{0.2cm} K_U(\boldsymbol{\xi}_i) = K_L(\boldsymbol{\lambda}_i,\boldsymbol{\xi}_i) \,.
\end{equation}
\subsection{Multi-agent problem}
All $\boldsymbol{\xi}_i$ are mutually exclusive meaning, $\boldsymbol{\xi}_i \cap \boldsymbol{\xi}_j = \emptyset, i \neq j$. The set of waypoint allocation for all gliders is denoted by $\boldsymbol{\Xi} = \{\boldsymbol{\xi}_i\}$.

The multi-agent problem is to find the jointly exhaustive allocations $\boldsymbol{\xi}_i$ with the least number of unvisited interest points. If there are multiple allocations with the same minimum number of unvisited interest points, we should find the one with least accumulated arclength of composite paths.
\begin{multline} \label{optupper}
\argmin_{\boldsymbol{\Xi}^*} \sum_{i=1}^{n_g} \sum_{ \forall j} l^f_{i,j}(\boldsymbol{\lambda}^{\boldsymbol{\xi}^*_i}_i)
\hspace{0.2cm} \text{s.t.} \hspace{0.2cm}
\boldsymbol{\Xi}^* \in \argmin_{\boldsymbol{\Xi}} \sum_{i=1}^{n_g} K_U(\boldsymbol{\xi}_i) \\
\text{such that } \bigcup_{i=1}^{n_g}\boldsymbol{\xi}_i = \{ip_{1},..,ip_{n_{ip}}\}.
\end{multline}
\subsection{Assumptions}
\begin{figure}[t]
	\centering
	\begin{minipage}{.25\textwidth}
		\centering
		\includegraphics[width=1.0\textwidth]{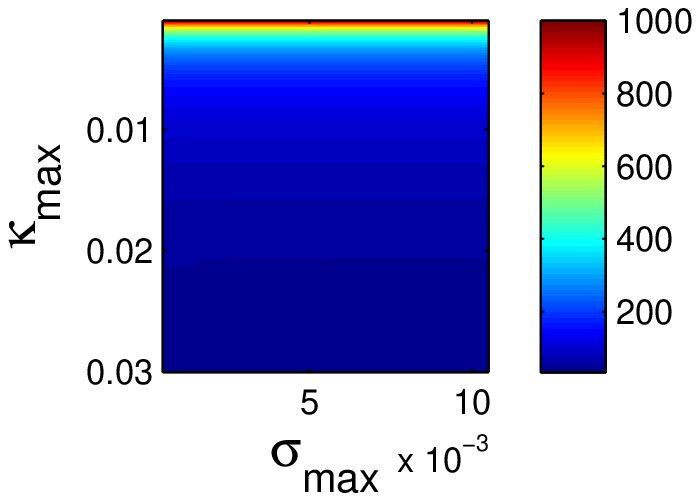}
		\caption{$R_T$ plotted against $\kappa_{max}$ and $\sigma_{max}$}
		\label{fig:R_T_versus_kappa_max}
	\end{minipage} \hspace{0.1cm}
	\begin{minipage}{.21\textwidth}
		\centering
		\includegraphics[width=1\textwidth]{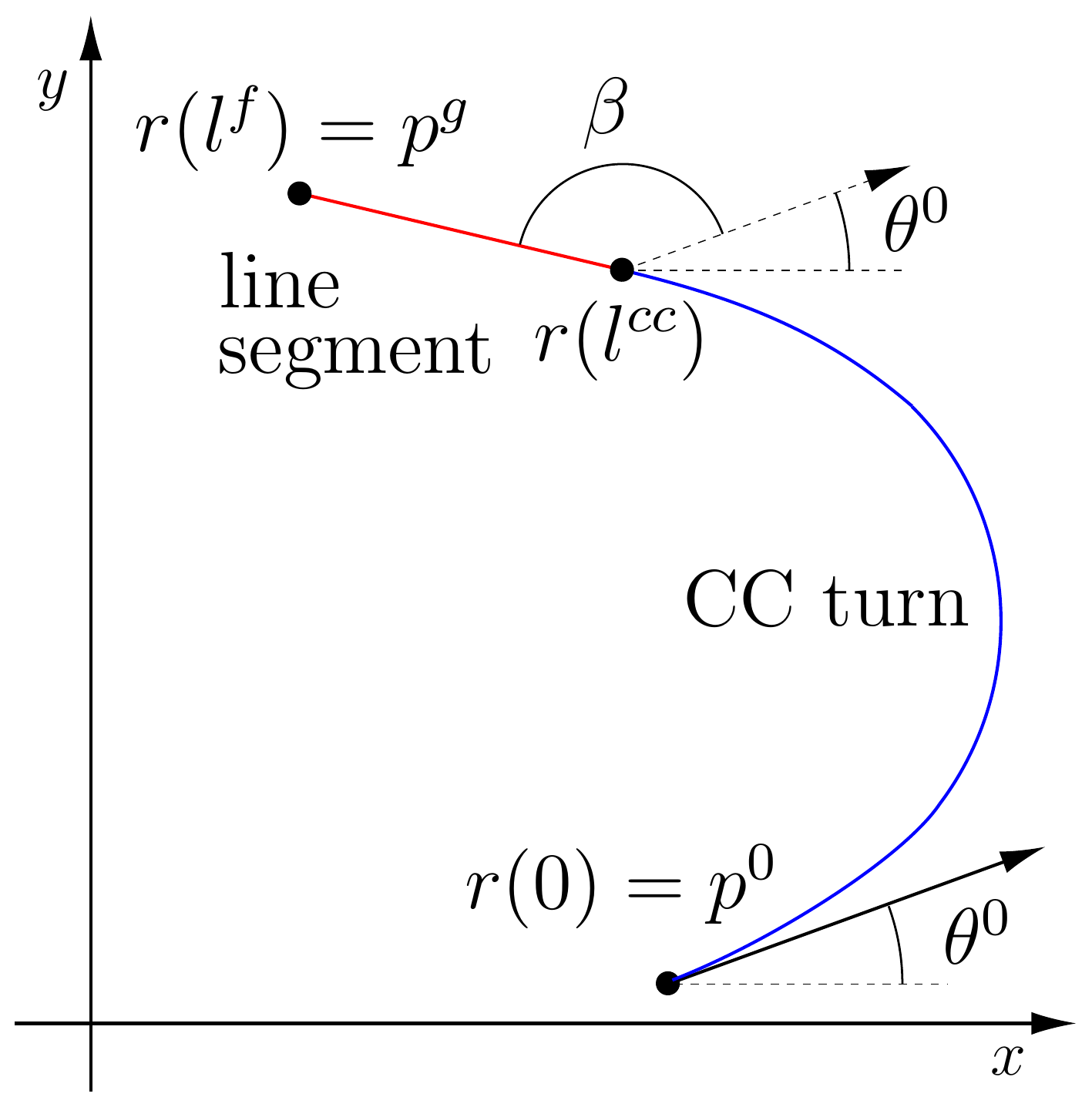}
		\caption{a leg in the composite path}
		\label{fig:ccleg}
	\end{minipage}
\end{figure}
\begin{asm} \label{asm:elementary}
	$\theta_{lim} \triangleq \kappa_{max}^2/\sigma_{max} < \pi$.
\end{asm}
\begin{asm} \label{asm:l_min>2R_T}
	$l_{min} > 2 R_T(\kappa_{max},\sigma_{max})$, where $R_T \in \mathbb{R}^+$ is defined in \cite{scheuer1997continuous} and $l_{min}$ is the smallest euclidean distance between any two waypoints, start points or final points of any glider.
\end{asm}
Assumption \ref{asm:elementary} is a carryover from \cite{scheuer1997continuous} and is needed for the legs to be feasible. Assumption \ref{asm:l_min>2R_T} is needed to make sure that the start and end point of the leg are not so close that the particular kind of paths introduced in \cite{scheuer1997continuous} become unachievable. These have been discussed in \cite{uzpath}. 

\section{Planning a composite path for a given visitation order} \label{sec:point2point}
The procedure to construct a feasible composite path for a given visitation order has been discussed in detail in our earlier work \cite{uzpath}. This section provides a brief summary. The composite path has a piece-wise linear curvature profile and it is designed to satisfy feasibility by construction. We start with how to construct a leg in the composite path and calculate its arclength. Then the procedure to construct the composite path is elaborated. This section is concluded by finding an upper bound on the ratio between arclength of a leg and the euclidean distance between its the start and the goal position. This upper bound is used later on in the paper.

A leg in the composite path provides a path from a start configuration (position and orientation) to a goal position. The leg has a piece-wise linear curvature profile. The leg is composed of two segments: (i) a Continuous Curvature (CC) turn that changes the orientation of the path by an angle $\beta$, and (ii) a line segment that completes the leg. The usage of CC turn guarantees that the curvature and sharpness of the leg are bounded, if assumptions \ref{asm:elementary} and \ref{asm:l_min>2R_T} are satisfied. A leg obtained by this methodology is shown in Figure \ref{fig:ccleg}.

\subsection{Constructing a Leg} \label{planusingccturns}

In this section we will focus on one leg denoted by $r(l)$. Subscripts $i$ and $j$ have been omitted for simplicity. The starting position of the leg is denoted by $r(0) = p^0$, starting orientation by $\theta(0) = \theta^0$ and goal position by $r(l^f) = p^g$. This is depicted in Figure \ref{fig:ccleg}. Without loss of generality we assume that $p^g$ lies on the left hand side of the start configuration.

The CC turn starts at $r(0)$ and ends at $r(l^{cc})$ where $l^{cc} < l^f$ as shown in Figure \ref{fig:ccleg}. The purpose of the CC turn is to change the orientation of the leg until it points to the final position. Therefore, the first step is to calculate the required change in orientation $\beta$. The value of $\beta$ is calculated in \cite{uzpath}. This is accomplished by using the fact that the CC turns satisfy two conditions. 1) The end points of a CC turn always exist on the circle $C^f_l$, and 2) the CC turn always makes an angle $\gamma$ with the tangent to the circle $C^f_l$ at these endpoints.

There are two types of CC turns depending on the value of $\beta$ and $\theta_{lim}$. For the case $\beta \in [0, \theta_{lim})$, the curvature profile of the CC turn is,
\begin{equation}
\kappa(l) = \left\{ 
\begin{array}{cc}
\sigma_e l & \text{for } 0 \leq l \leq l^{cc}/2 \\
\sigma_e (l^{cc} - l)  & \text{for } l^{cc}/2 < l \leq l^{cc}
\end{array}
\right.
\end{equation}
where the expression for $\sigma_e$ is given in \cite{scheuer1997continuous}. Otherwise, if $\beta \in [\theta_{lim}, 2 \pi]$, the curvature profile of the CC turn is,
\begin{multline}
\kappa(l) = \left\{ 
\begin{array}{ll}
\sigma_{max} l & \text{for } 0 \leq l \leq l^{cl} \\
\kappa_{max}  & \text{for } l^{cl} < l \leq l^{cc} - l^{cl} \\
\sigma_{max}(l^{cc}-l) & \text{for } l^{cc} - l^{cl} < l \leq l^{cc}
\end{array}
\right.
\end{multline}
where $l^{cl} = \kappa_{max}/\sigma_{max}$.

The arclength of each CC turn (as given in \cite{scheuer1997continuous}), is,
\begin{equation} \label{l_cc}
l^{cc}(\beta) = \left\{
\begin{array}{ll}
2 \sqrt{\beta/\sigma_e} & \text{if } \beta \in [0, \theta_{lim}) \\
\beta/\kappa_{max} + \kappa_{max}/\sigma_{max} & \text{if } \beta \in [\theta_{lim}, 2 \pi] \,.
\end{array}
\right.
\end{equation}
The orientation at the end of the CC turn is,
\begin{equation}
\theta(l^{cc}) = \theta(0) + \beta \,.
\end{equation}
The next part in the leg is a line segment which starts at $r(l^{cc})$ and ends at $r(l^f)$. The orientation, curvature and sharpness of the line segment is,
\begin{equation}
\theta(l) = \theta(l^{cc}), \hspace{0.2cm} \kappa(l) = \sigma(l) = 0 \text{ for } l^{cc} < l \leq l^f
\end{equation}
The total arc length of the leg $l^f$ has been derived in \cite{uzpath}. Now, each leg $r(l)$ can be constructed by using the curvature profile $\kappa(l)$ of the leg and its initial conditions.

\subsection{Composite Path} \label{sec:comppath}
For every combination of glider and visitation order there is a unique composite path. The composite path is constructed by using the boundary conditions \eqref{boundarycond}-\eqref{boundarycond2} along with the procedure to construct the leg. The curvature constraint in equation \eqref{contposorientcurv} is automatically satisfied since curvature and sharpness of a leg at $l=\{0, l^{cc}, l^f\}$ is $0$. A composite path satisfies the feasibility constraint by design and the validity of the path can be determined by calculating its arclength. The arclength of each composite path is calculated by summing the arclengths of its constituent legs.

\subsection{Upper bound on ratio between length of the leg $l^f$ and euclidean distance $l^e$}
In this section we obtain an upper bound on the ratio $l^f/l^e$, where $l^e$ is the euclidean distance between $p^g$ and $p^0$.
\begin{equation} \label{l^e}
	l^e \left( p^g, p^0 \right) = ||p^g - p^0 ||_2  \,,
\end{equation}
The following theorem presents the result. This result is used to guarantee optimality of the upper level graph search.
\begin{lemma} \label{thm:epsilon1}
	If assumptions \ref{asm:elementary} and \ref{asm:l_min>2R_T} are met, the ratio $l^f/l^e \leq \mathcal{R}_{max}$, which is given by,
	\begin{multline} \label{R_max}
		\mathcal{R}_{max} = \frac{\sqrt{(l_{min} + R_T)^2 - R_M^2}}{l_{min}} + \mathcal{R}^{cc}_{max} \\
		\text{where } \mathcal{R}^{cc}_{max} = \frac{\max \left(\frac{\beta_{max}(l_{min})}{\kappa_{max}} + \frac{\kappa_{max}}{\sigma_{max}} , \frac{4.66 R_T}{l_{min}} \right)+ R_T}{l_{min}}
	\end{multline}
\end{lemma}
\begin{proof}
The proof of the theorem is given in the Appendix.
\end{proof}

\section{Single-agent problem} \label{sec:singleglider}
The aim of this section is to solve the single-agent problem. This corresponds to finding a composite path which i) is valid and ii) maximizes the number of (allocated) interest points visited by the glider. This is achieved by choosing the best visitation order, since the composite path is uniquely determined by its corresponding visitation order.
The problem of finding the best visitation order is a constrained integer programming problem and we formulate it as a graph search. This is called the lower level graph and its denoted by $\Gamma_L$. It is dependent on the allocation (of waypoints) $\boldsymbol{\xi}_i$ and denoted by $\Gamma_L(\boldsymbol{\xi}_i)$. Each node in the graph represents a particular visitation order.

This section is organized in the following manner. First, the construction of $\Gamma_L$ and the notation associated with it is explained. Next, the optimality condition for single-agent problem along with a reformulation of this condition is restated in the current notation. After that the  new concepts of \emph{ideality} and \emph{weak validity} are introduced. These notions are used to obtain optimality guarantees in the next section. Finally, the algorithm to find the best visitation order, called lower level graph search, is presented. 
\subsection{Graph $\Gamma_L(\boldsymbol{\xi}_i)$}
Every node in the graph $\Gamma_L(\boldsymbol{\xi}_i)$ is a visitation order $\boldsymbol{\lambda}_i$. Each node $\boldsymbol{\lambda}_i$ has a composite path made up of legs $r_{i,j}(\boldsymbol{\lambda}_i)$ and a height profile $h_{i,j}(\boldsymbol{\lambda}_i)$ as described in Section \ref{sec:probstat}. The root (starting) node of the graph is denoted by $\boldsymbol{\lambda}^0_i$. The order associated with this node is empty $\boldsymbol{\lambda}^0_i = \{\}$, and hence it has no composite path. The graph search is terminated when a goal node is encountered. A node is considered a goal node if its last waypoint is the final position of the glider. The set of goal nodes is denoted by $\Phi^f_L$. Goal nodes do not have children.

The child of node $\boldsymbol{\lambda}_i$ is obtained, by appending a leg to the end of composite path $r_{i,j}(\boldsymbol{\lambda}_i)$. This is done by adding a thermal, an unvisited allocated interest point or the final position (of the glider) to the end of $\boldsymbol{\lambda}_i$. The set of children of $\boldsymbol{\lambda}_i$, $Children(\boldsymbol{\lambda}_i)$ is defined as,
\begin{multline} \label{eq5}
\boldsymbol{\lambda'}_i \in Children(\boldsymbol{\lambda}_i) \text{ only if } \boldsymbol{\lambda'}_i = \{\lambda_{i,1},.., \lambda_{i,n(\boldsymbol{\lambda}_i)}, \mu \} \\
\text{ s.t. } \mu \in \left\{\boldsymbol{\xi}_i \cup \{t_1,..,t_{n_t}\} \cup \{f\} \right\} \setminus \boldsymbol{\lambda}_i
\end{multline}
where $n()$ is the cardinality operator. If $\boldsymbol{\lambda'}_i \in Children(\boldsymbol{\lambda}_i)$ then $\boldsymbol{\lambda}_i = Parent(\boldsymbol{\lambda'}_i)$. 

Now we introduce some values associated with each $\boldsymbol{\lambda}_i$.
\subsubsection{Arc length of a node}
The arc length of the node, denoted by $S_L(\boldsymbol{\lambda}_i)$, is the total arc length of the composite path associated with $\boldsymbol{\lambda}_i$.
\begin{equation} \label{actcostlower}
	S_L(\boldsymbol{\lambda}_i) = \sum\limits_{j = 1}^{n(\boldsymbol{\lambda}_i)} l^f_{i,j} \,,
\end{equation}
where $l^f_{i,j}$ is the arclength of the $j$th leg of the glider's composite path.

\subsubsection{Cost of a node}
Another quantity associated with the a node $\boldsymbol{\lambda}_i$ is its cost $V_L$. The cost of a node determines the order of expansion of the node in the lower level graph search. Nodes with less cost are chosen first. It is defined as,
\begin{equation} \label{costoflambda}
V_L(\boldsymbol{\lambda}_i,\boldsymbol{\xi}_i) = \left\{
\begin{array}{ll}
S_L(\boldsymbol{\lambda}_i) & \text{ if } \boldsymbol{\lambda}_i \notin \Phi^f_L \\
S_L(\boldsymbol{\lambda}_i) + K_L(\boldsymbol{\lambda}_i)P_L & \text{ if } \boldsymbol{\lambda}_i \in \Phi^f_L \,,
\end{array}
\right.
\end{equation}
\begin{equation} \label{eq6}
	\text {where } P_L = \left( h^0_{i} +  \sum_{k = 1}^{n_t} {h_{t_k}} + 1 \right) / \tan(\gamma_{d_{min}}) \,.
\end{equation}

The symbol $\boldsymbol{\xi}_i$ is omitted from the expression $K_L(\boldsymbol{\lambda}_i)$ since it is clear from context. The $K_L(\boldsymbol{\lambda}_i)P_L$ term, is a penalty term which penalizes those goal nodes which have more unvisited allocated interest points. This makes them less likely to be expanded and makes the graph search optimal.

\subsubsection{Validity of a node}
We define a set of nodes called the set of \emph{valid} nodes, denoted by $\Phi^v_L$. A node is considered valid if its corresponding composite path satisfies the validity constraint \eqref{cons:h}. This set can be expressed as,
\begin{equation} \label{validcheck}
\Phi^v_L = \{\boldsymbol{\lambda}_i | S_L(\boldsymbol{\lambda}_i) < S^{max}_L(\boldsymbol{\lambda}_i) \hspace{0.1cm} \& \hspace{0.1cm}  Parent(\boldsymbol{\lambda}_i) \in \Phi^v_L \} \,.
\end{equation}

The quantity $S^{max}_L(\boldsymbol{\lambda}_i)$ is the upper bound on the arclength of composite path.
\begin{equation} \label{upperboundarclength}
S^{max}_L(\boldsymbol{\lambda}_i) = \left(h^0_{i} + \sum\limits_{l = 1}^{n_t} \mathbbm{1}_{\boldsymbol{\lambda}_i}(t_l)h_{t_l} \right)/\tan(\gamma_{d_{min}}) \,.
\end{equation}

\subsection{Optimality condition}
The optimality condition of the single-agent problem as expressed in equation \eqref{objfuncsingle} can be restated in the current notation as,
\begin{equation} \label{objfuncllorig}
\argmin_{\boldsymbol{\lambda}^*_i} S_L(\boldsymbol{\lambda}^*_i) \hspace{0.1cm} \text{s.t.} \hspace{0.1cm} \boldsymbol{\lambda}^*_i \in \argmin_{\boldsymbol{\lambda}_i}{K_L(\boldsymbol{\lambda}_i)} \cap \Phi^f_L \cap \Phi^v_L
\end{equation}

This conditions involves first finding the set of all valid goal nodes which have least number of unvisited interest points. Then we need to find the node with the least arclength from this set. Now consider another condition where we want to find a node with the least cost $V_L$ from the set of all valid goal nodes. This is expressed below,
\begin{equation} \label{objfuncllnew}
\argmin_{\boldsymbol{\lambda}^*_i} V_L(\boldsymbol{\lambda}^*_i) \hspace{0.1cm} \text{s.t.} \hspace{0.1cm} \boldsymbol{\lambda}^*_i \in \Phi^f_L \cap \Phi^v_L
\end{equation}

In theorem \eqref{lem4} we show that if a node satisfies condition \eqref{objfuncllnew} then it will also satisfy condition \eqref{objfuncllorig}. For this to be true, the penalty term $P_L$ has to be an upper bound on the distance the glider can travel. This can be expressed as the upper bound on height of the glider divided by $\tan(\gamma_{d_{min}})$. An upper bound on height of the glider is the starting height of the glider plus the sum of the heights of all thermals $h^0_{i} + \sum {h_{t_k}}$. This idea is akin to the idea of Lagrangian relaxation.

\begin{theorem} \label{lem4}
A node in graph $\Gamma_L$ that satisfies condition \eqref{objfuncllnew} will also satisfy condition \eqref{objfuncllorig}.
\end{theorem}
\begin{proof}
	
We prove this lemma in two parts. Let us assume that the node $\boldsymbol{\hat\lambda}_i$ satisfies condition \eqref{objfuncllnew}. 

First, we prove that $\boldsymbol{\hat\lambda}_i$ has the least unvisited interest points in the set $\Phi^f_l \cap \Phi^v_l$. This is a proof by contradiction. Suppose, $\exists \boldsymbol{\tilde\lambda}_i \in \Phi^f_L \cap \Phi^v_l$, such that $K_L(\boldsymbol{\tilde\lambda}_i)<K_L(\boldsymbol{\hat\lambda}_i)$.
\begin{equation*}
V_L(\boldsymbol{\hat\lambda}_i) \leq V_L(\boldsymbol{\tilde\lambda}_i) \,, \hspace{0.1cm} \text{(because of condition \eqref{objfuncllnew})}
\end{equation*}
\begin{equation*}
S_L(\boldsymbol{\hat\lambda}_i) + K_L(\boldsymbol{\hat\lambda}_i)P_L \leq S_L(\boldsymbol{\tilde\lambda}_i) + K_L(\boldsymbol{\tilde\lambda}_i)P_L \,,
\end{equation*}
\begin{equation*}
(K_L(\boldsymbol{\hat\lambda}_i) - K_L(\boldsymbol{\tilde\lambda}_i))P_L \leq S_L(\boldsymbol{\tilde\lambda}_i) - S_L(\boldsymbol{\hat\lambda}_i) \,,
\end{equation*}
\begin{equation*}
(K_L(\boldsymbol{\hat\lambda}_i) - K_L(\boldsymbol{\tilde\lambda}_i))P_L \leq \left(h^0_{i} + \sum^{n_t}_{i = 1} {h_{t_k}} \right) / \tan(\gamma_{d_{min}}) \,.
\end{equation*}
The last step was made possible by the fact that $\left(h^0_{i} + \sum^{n_t}_{i = 1} {h_{t_k}} \right) / \tan(\gamma_{d_{min}})$ is an upper bound on $S_L(\boldsymbol{\lambda}_i)$ for any valid node $\boldsymbol{\lambda}_i$. Using the definition of $P_L$ and the fact that $K_L(\boldsymbol{\hat\lambda}_i) - K_L(\boldsymbol{\tilde\lambda}_i) \geq 1$, we arrive at a contradiction.
	
Secondly, we prove that $\boldsymbol{\hat\lambda}_i$ has the least arclength of all nodes with similar $K_L$ in the set $\Phi^f_l \cap \Phi^v_l$. Assume another node $\boldsymbol{\dot\lambda}_i \in \Phi^f_L \cap \Phi^v_l$, such that $K_L(\boldsymbol{\dot\lambda}_i) = K_L(\boldsymbol{\hat\lambda}_i)$.
\begin{equation*}
V_L(\boldsymbol{\hat\lambda}_i) \leq V_L(\boldsymbol{\dot\lambda}_i) \,,
\end{equation*}
\begin{equation*}
S_L(\boldsymbol{\hat\lambda}_i) + K_L(\boldsymbol{\hat\lambda}_i)P_L \leq S_L(\boldsymbol{\dot\lambda}_i) + K_L(\boldsymbol{\dot\lambda}_i)P_L \,,
\end{equation*}
\begin{equation*}
S_L(\boldsymbol{\hat\lambda}_i) \leq S_L(\boldsymbol{\dot\lambda}_i) \,.
\end{equation*}
Hence $\boldsymbol{\hat\lambda}_i$ has the minimum arclength for all nodes in $\Phi^f_l \cap \Phi^v_l$ with number of unvisited interest points equal to $K_L(\boldsymbol{\hat\lambda}_i)$.
\end{proof}
The reason for using the second condition is that finding a node that satisfies the second condition is computationally less expensive. Condition \eqref{objfuncllorig} involves taking a minimum over a set two times whereas condition \eqref{objfuncllnew} requires it once. Moreover, the second condition can be satisfied by running a uniform cost search over the graph $\Gamma_L$. This procedure is called the lower level graph search and is explained in section \ref{lowerlevelgraphsearch}. Before delving into the lower level graph search we introduce two new concepts of \emph{ideality} and \emph{weakness} which will be used in the following sections.

\subsection{Ideality and weakness}
First we define the concept of an \emph{ideal path} of a glider for a given waypoint visitation order. Like a composite path, this path satisfies the boundary conditions \eqref{boundarycond}-\eqref{contposorientcurv} and feasibility conditions \eqref{cons:kappa_max}-\eqref{cons:sigma_max}. But it has the least arclength over all paths that satisfy these condition. The ideal path associated with the node $\boldsymbol{\lambda}_i$ is denoted by $r^*_{i,j}(\boldsymbol{\lambda}_i )$. Such a path is guaranteed to exist due to Filippov's Existence Theorem \cite{liberzon2012calculus}.

Similar to the concept of arclength of a node we also have the \emph{ideal arclength} of a node. This is the total arclength of the ideal path $r^*_{i,j}$ and it is denoted by $S^*_L(\boldsymbol{\lambda}_i)$. Likewise each node also has an \emph{ideal cost} $V^*_L(\boldsymbol{\lambda}_i)$. This is given as,
\begin{equation}
V^*_L(\boldsymbol{\lambda}_i,\boldsymbol{\xi}_i) = \left\{
\begin{array}{ll}
S^*_L(\boldsymbol{\lambda}_i) & \text{if } \boldsymbol{\lambda}_i \notin \Phi^f_L \\
S^*_L(\boldsymbol{\lambda}_i) + K_L(\boldsymbol{\lambda}_i,\boldsymbol{\xi}_i)P_L & \text{if } \boldsymbol{\lambda}_i \in \Phi^f_L \,,
\end{array}
\right.
\end{equation}

When clear from context the $\boldsymbol{\xi}_i$ is omitted for simplicity. Similarly we have the concept of \emph{ideal validity}. A node $\boldsymbol{\lambda}_i$ is considered \emph{ideally} valid if its ideal path is valid. The set of ideally valid nodes is denoted by $\Phi^{iv}_L$. This can be rigorously be defined as,
\begin{equation} \label{idealvalidcheck}
\Phi^{iv}_L = \{\boldsymbol{\lambda}_i | S^*_L(\boldsymbol{\lambda}_i) < S^{max}_L(\boldsymbol{\lambda}_i) \hspace{0.1cm} \& \hspace{0.1cm} Parent(\boldsymbol{\lambda}_i) \in \Phi^{iv}_L \} \,.
\end{equation}

Finally, we introduce the concept of \emph{ideal optimality}. This is similar to the concept of optimality as shown in conditions \eqref{objfuncllorig} and \eqref{objfuncllnew}. A node is ideally optimal if it satisfies the following condition.
\begin{equation} \label{objfuncllideal}
\argmin_{\boldsymbol{\lambda}^*_i} V^*_L(\boldsymbol{\lambda}^*_i) \hspace{0.1cm} \text{s.t.} \hspace{0.1cm} \boldsymbol{\lambda}^*_i \in \Phi^f_L \cap \Phi^{iv}_L
\end{equation}
This means that a node is ideally optimal if it minimizes ideal cost over all ideally valid goal nodes. Similar to the regular optimality this also means that its ideal path visits the most interest points and has the least arclength. This is formalized in the following lemma.
\begin{lemma} \label{lem4_ideal}
If a node satisfies ideal optimality condition \eqref{objfuncllideal} then it will also satisfy the following condition,
\begin{equation}
\argmin_{\boldsymbol{\lambda}^*_i} S^*_L(\boldsymbol{\lambda}^*_i) \hspace{0.1cm} \text{s.t.} \hspace{0.1cm} \boldsymbol{\lambda}^*_i \in \{\boldsymbol{\lambda}_i | \argmin{K_L(\boldsymbol{\lambda}_i)}\} \cap \Phi^f_L \cap \Phi^v_L
\end{equation}
\end{lemma}
\begin{proof}
The proof for this lemma is similar to the proof for theorem \ref{lem4}. It can be easily obtained by replacing $S_L$, $V_L$ and $\Phi^v_L$ with $S^*_L$, $V^*_L$ and $\Phi^{iv}_L$ respectively.
\end{proof}

Notice that the ideal path, arclength, cost, validity and optimality of a node is unknown.

Now we move on to determinable notions. The first one is \emph{weak cost} of a node. This is denoted by $\underbar{V}_L(\boldsymbol{\lambda}_i)$ and defined as,
\begin{equation} \label{eq:defweakcost}
\underbar{V}_L(\boldsymbol{\lambda}_i,\boldsymbol{\xi}_i) = \left\{
\begin{array}{ll}
S_L(\boldsymbol{\lambda}_i)/\mathcal{R}_{max} & \text{if } \boldsymbol{\lambda}_i \notin \Phi^f_L \\
S_L(\boldsymbol{\lambda}_i)/\mathcal{R}_{max} + K_L(\boldsymbol{\lambda}_i)P_L & \text{if } \boldsymbol{\lambda}_i \in \Phi^f_L \,,
\end{array}
\right.
\end{equation}

When clear from context the $\boldsymbol{\xi}_i$ is omitted for simplicity. Another notion related to validity is called \emph{weak} validity. The set of weakly valid nodes is denoted by $\phi^{wv}_L$. This set is defined in the following way,
\begin{equation} \label{weakvalidcheck}
\Phi^{wv}_L = \{\phi | \frac{S_L(\boldsymbol{\lambda}_i)}{\mathcal{R}_{max}}  < S^{max}_L(\boldsymbol{\lambda}_i) \hspace{0.1cm} \& \hspace{0.1cm} Parent(\boldsymbol{\lambda}_i) \in \Phi^{wv}_L \} \,.
\end{equation}
Finally, we introduce the concept of \emph{weak optimality}. A node is weakly optimal if it has the least weak cost in all the weakly valid goal nodes. This expressed as follows.
\begin{equation} \label{objfuncllweak}
\boldsymbol{\underbar{$\lambda$}}_i = \argmin_{\boldsymbol{\lambda}^*_i} \underbar{V}_L (\boldsymbol{\lambda}^*_i) \hspace{0.1cm} \text{s.t.} \hspace{0.1cm} \boldsymbol{\lambda}^*_i \in \Phi^f_L \cap \Phi^{wv}_L \,.
\end{equation}
Notice that weakly valid and optimal nodes are determinable as opposed to the ideal path, arclength, cost, validity and optimality.

We end this subsection with some results which relate these new notions to the earlier concepts. This will be helpful in proving the optimality of the Branch\&Bound algorithm. Lemma \ref{lem:validsets} states that if a node is valid it will also be ideally valid, and if a node is ideally valid it will also be weakly valid.
\begin{lemma} \label{lem:validsets}
$\Phi^v_L \subset \Phi^{iv}_L \subset \Phi^{wv}_L$.
\end{lemma}
\begin{proof}
As shown in lemma \ref{thm:epsilon1}, any leg in a composite path constructed using the procedure outlined in section \ref{sec:point2point} will have an arclength $l^f_{i,j} \leq \mathcal{R}_{max} l^e_{i,j}$, where $l^e_{i,j}$ is the euclidean distance between the start and end positions of the leg. Moreover, its also quite obvious that $l^e_{i,j} \leq l^f_{i,j}$. Hence, we get the inequality,
\begin{equation*}
\frac{l^f_{i,j}}{\mathcal{R}_{max}} \leq l^e_{i,j} \leq l^f_{i,j} \leq \mathcal{R}_{max} l^e_{i,j} \,,
\end{equation*}
If we sum this over all the legs of the composite path we get the expression,
\begin{equation*}
\frac{S_L(\boldsymbol{\lambda}_i)}{\mathcal{R}_{max}} \leq \sum_{\forall j} l^e_{i,j} \leq S_L(\boldsymbol{\lambda}_i) \leq \mathcal{R}_{max} \sum_{\forall j}l^e_{i,j} \,,
\end{equation*}
Now we know that $S^*_L(\boldsymbol{\lambda}_i) \leq S_L(\boldsymbol{\lambda}_i)$ by definition and $\sum_{\forall j} l^e_{i,j} \leq S^*_L(\boldsymbol{\lambda}_i)$. Using this we get,
\begin{equation*}
\frac{S_L(\boldsymbol{\lambda}_i)}{\mathcal{R}_{max}} \leq \sum_{\forall j} l^e_{i,j} \leq S^*_L(\boldsymbol{\lambda}_i)  \leq S_L(\boldsymbol{\lambda}_i) \leq \mathcal{R}_{max} \sum_{\forall j}l^e_{i,j} \,,
\end{equation*}
This can be shortened to,
\begin{equation} \label{SL/R<S*L<SL}
\frac{S_L(\boldsymbol{\lambda}_i)}{\mathcal{R}_{max}} \leq S^*_L(\boldsymbol{\lambda}_i)  \leq S_L(\boldsymbol{\lambda}_i) \,,
\end{equation}

From here it is easy to see that if a node satisfies validity it will definitely satisfy ideal validity and if it satisfies ideal validity it will definitely satisfy weak validity. Hence, $\Phi^v_L \subset \Phi^{iv}_L \subset \Phi^{wv}_L$.
\end{proof}

Now we introduce some lemmas which relate cost of a node with its ideal and weak costs.

\begin{lemma} \label{lem:Vund<V*}
If the node $\boldsymbol{\underbar{$\lambda$}}_i$ satisfies weak optimality condition \eqref{objfuncllweak} and node $\boldsymbol{\lambda}^*_i$ satisfies ideal optimality condition \eqref{objfuncllideal}, then the weak cost of $\boldsymbol{\underbar{$\lambda$}}_i$ will be a lower bound on the ideal cost of $\boldsymbol{\lambda}^*_i$,
\begin{equation}
\underbar{V}_L(\boldsymbol{\underbar{$\lambda$}}_i) \leq V^*_L(\boldsymbol{\lambda}^*_i)
\end{equation}
\end{lemma}
\begin{proof}
As $\boldsymbol{\underbar{$\lambda$}}_i$ minimizes $\underbar{V}_L $ over all nodes in the set $\Phi^f_L \cap \Phi^{wv}_L$,
\begin{equation*}
\underbar{V}_L(\boldsymbol{\underbar{$\lambda$}}_i) \leq \underbar{V}_L(\boldsymbol{\lambda}^*_i)
\end{equation*}
since $\boldsymbol{\lambda}^*_i \in \Phi^f_L \cap \Phi^{iv}_L$ and $\Phi^f_L \cap \Phi^{iv}_L$ is a subset of $\Phi^f_L \cap \Phi^{wv}_L$.

Equation \eqref{SL/R<S*L<SL} says $\frac{S_L(\boldsymbol{\lambda}_i)}{\mathcal{R}_{max}} < S^*_L(\boldsymbol{\lambda}_i) < S_L(\boldsymbol{\lambda}_i), \forall \phi$ hence $\underbar{V}_L(\boldsymbol{\lambda}_i) < V^*_L(\phi) < V_L(\boldsymbol{\lambda}_i)$ for any node. This leads us to,
\begin{equation*}
\underbar{V}_L(\boldsymbol{\underbar{$\lambda$}}_i) \leq \underbar{V}_L(\boldsymbol{\lambda}^*_i) \leq V^*_L(\boldsymbol{\lambda}^*_i)
\end{equation*}
\end{proof}

\begin{lemma} \label{lem:V*<V}
If the node $\boldsymbol{\lambda'}_i$ satisfies optimality condition \eqref{objfuncllnew} and node $\boldsymbol{\lambda}^*_i$ satisfies ideal optimality condition \eqref{objfuncllideal}, then the ideal cost of $\boldsymbol{\lambda}^*_i$ will be a lower bound for the cost of $\boldsymbol{\lambda'}_i$, namely,
\begin{equation}
V^*_L(\boldsymbol{\lambda}^*_i) \leq V_L (\boldsymbol{\lambda'}_i)
\end{equation}
\end{lemma}
\begin{proof}
From the ideal optimality condition we know that. 
\begin{equation*}
V^*_L (\boldsymbol{\lambda}^*_i) \leq V^*_L (\boldsymbol{\lambda}_i), \hspace{0.1cm} \forall \boldsymbol{\lambda}_i \in \Phi^{iv}_L \cap \Phi^f_L
\end{equation*}

We also know that $\boldsymbol{\lambda'}_i$ is a valid goal node in $\Gamma_L (\boldsymbol{\xi}^{\phi'_u}_l )$ and all valid goal nodes are also ideally valid goal nodes. Hence,
\begin{equation*}
V^*_L (\boldsymbol{\lambda}^*_i) \leq V^*_L (\boldsymbol{\lambda'}_i)
\end{equation*}

From equation \eqref{SL/R<S*L<SL} we know that $V^*_L(\phi) \leq V_L(\phi)$ for any node, hence,
\begin{equation*}
V^*_L (\boldsymbol{\lambda}^*_i) \leq V^*_L (\boldsymbol{\lambda'}_i) \leq V_L (\boldsymbol{\lambda'}_i)
\end{equation*}
\end{proof}
\subsection{Lower level graph search over $\Gamma_L$} \label{lowerlevelgraphsearch}
\begin{algorithm}[h]
	\caption {Lower Level Graph Search}
	\label{alg2}
	\begin{algorithmic}[1]
		\REQUIRE $i,p^0_{i},h^0_{i},p^f_{i},\boldsymbol{\xi}_i $
		\STATE $\Omega = \{\boldsymbol{\lambda}^0_i\}, \hspace{0.2cm} \boldsymbol{\lambda}_i := \boldsymbol{\lambda}^0_i$
		\WHILE {$\boldsymbol{\lambda}_i \notin \Phi^f_L$}
		\STATE $\textrm{remove } \boldsymbol{\lambda}_i \textrm{ from } \Omega \textrm{ with the smallest } V_L(\boldsymbol{\lambda}_i)$
		\FORALL {$\boldsymbol{\lambda'}_i \in Children(\boldsymbol{\lambda}_i)$}
		\STATE calculate $S_L(\boldsymbol{\lambda'}_i)$, $V_L(\boldsymbol{\lambda'}_i)$ and $\underbar{V}_L(\boldsymbol{\lambda'}_i)$
		\IF {$S_L(\boldsymbol{\lambda'}_i)$ satisfies validity \eqref{validcheck}}
		\STATE Insert $\boldsymbol{\lambda'}_i$ into $\Omega$ with cost $V_L(\boldsymbol{\lambda'}_i)$
		\ELSIF {$S_L(\boldsymbol{\lambda'}_i)$ satisfies weak validity \eqref{weakvalidcheck}}
		\STATE Insert $\boldsymbol{\lambda'}_i$ into $\underline\Omega $ with cost $\underbar{V}_L(\boldsymbol{\lambda'}_i)$
		\ENDIF
		\ENDFOR
		\ENDWHILE
		\STATE $\boldsymbol{\hat\lambda}_i := \boldsymbol{\lambda}_i$
		\STATE \underbar{$\boldsymbol{\lambda}_i$} $:=$ node with smallest $\underbar{V}_L$ in $\underline\Omega$
		\WHILE {\underbar{$\boldsymbol{\lambda}_i$} $\notin \Phi^f_L$}
		\STATE remove \underbar{$\boldsymbol{\lambda}_i$} from $\underline\Omega$ with the smallest $\underbar{V}_L(\underbar{$\boldsymbol{\lambda}_i$})$
		\FORALL {\underbar{$\boldsymbol{\lambda}_i$}$' \in Children(\underbar{$\boldsymbol{\lambda}_i$})$ that do not exist}
		\STATE calculate $S_L(\underbar{$\boldsymbol{\lambda}_i$}')$ and $\underbar{V}_L(\underbar{$\boldsymbol{\lambda}_i$}')$
		\IF {$S_L(\boldsymbol{\lambda'}_i)$ satisfies weak validity \eqref{weakvalidcheck}}
		\STATE Insert $\underbar{$\boldsymbol{\lambda}_i$}'$ into $\underline\Omega $ with cost $\underbar{V}_L(\underbar{$\boldsymbol{\lambda}_i$}')$
		\ENDIF
		\ENDFOR
		\ENDWHILE
		\RETURN $\boldsymbol{\hat\lambda}_i, S_L(\boldsymbol{\hat\lambda}_i), K_L(\boldsymbol{\hat\lambda}_i), S_L(\underbar{$\boldsymbol{\lambda}_i$}), K_L(\underbar{$\boldsymbol{\lambda}_i$})$
	\end{algorithmic}
\end{algorithm}

Algorithm \ref{alg2} presents the lower level graph search over graph $\Gamma_L$. The algorithm is made up of two parts. The first part (lines 1-13) finds the node that satisfies the single-agent optimality condition \eqref{objfuncllnew}. Whereas the second part (lines 14-24) finds the node which satisfies the weak optimality condition \eqref{objfuncllweak}. This latter part finds a node which helps us in the Branch\&Bound algorithm.

The first part is a uniform cost graph search over the set $\Phi^v_L$. The search finds a valid goal node which minimizes the cost function $V_L$. For a graph search to be optimal, meaning it finds the node which satisfies the optimality condition, the cost of a child node should be greater than or equal to the cost of its parent node, as shown in the book \cite{russell1995modern}. This is called the monotonicity of cost. For algorithm \ref{alg2} this is shown easily. Consider a node $\boldsymbol{\lambda}_i$ and it child $\boldsymbol{\lambda'}_i$. If the $\boldsymbol{\lambda'}_i \notin \Phi^f_L$, $V(\boldsymbol{\lambda}_i) = S(\boldsymbol{\lambda}_i)$ and $V(\boldsymbol{\lambda'}_i) = S(\boldsymbol{\lambda'}_i)$. As we know that $\boldsymbol{\lambda'}_i$ is obtained by adding a leg to the composite path of $\boldsymbol{\lambda}_i$, $S(\boldsymbol{\lambda'}_i) > S(\boldsymbol{\lambda}_i)$, which leads to, $V(\boldsymbol{\lambda'}_i) > V(\boldsymbol{\lambda}_i)$. If the $\boldsymbol{\lambda'}_i \in \Phi^f_L$, using similar reasoning we arrive at the same conclusion. Hence, $V(\boldsymbol{\lambda'}_i) > V(\boldsymbol{\lambda}_i)$.

The second part of the algorithm performs a uniform cost graph search over the set $\Phi^{wv}_L$. It finds a node which satisfies the weak optimality condition. Since the weak cost satisfies monotonicity property this part is also optimal. This node and the values associated with it are used in the Branch\&Bound algorithm.

In the algorithm \ref{alg2} $Nodes(\Gamma_{L})$ refers to the set of all possible nodes in $\Gamma_{L}$. $\Omega$ is the open set (the set of nodes yet to be expanded) but is only limited to valid nodes. Whereas $\underline\Omega$ is also an open set but is limited to nodes which are not valid but satisfy weak validity.

\section{Multi-agent problem} \label{sec:multiglider}
In this section we solve the multi-agent problem. This corresponds to finding a feasible and valid composite path for each glider such that the interest points visited are maximized. We decompose this problem into many single-agent problems by allocating a set of interest points to each glider. The single-agent problem is solved by the lower level graph search as described in section \ref{sec:singleglider}. The problem now is to find the best allocation of interest points for each glider. This section presents two ways of achieving this objective. The first one is a brute force approach as presented in our earlier work \cite{uzpath}. The second is a Branch\&Bound type graph search, called the lower level graph search. The reason for using Branch\&Bound instead of uniform cost graph search is that we were unable to prove the monotonicity property in the upper level graph. The upper level graph search is shown to be faster than the brute force search via simulations. 

The section starts with introducing some values associated with each set of allocations $\boldsymbol\Xi = \{ \boldsymbol\xi_i \}$ and its notation. $\boldsymbol\Xi$ refers to the set of allocations and $\boldsymbol\xi_i$ is the allocation to a particular glider $i$. Secondly, the multi-agent optimality condition is restated along with its reformulation in the current notation. After that brute force search is presented. This is followed by the construction of the upper level graph $\Gamma_U$. Finally, the Branch\&Bound algorithm along with its optimality guarantees is presented.

\subsection{Values associated with each $\boldsymbol\xi_i$ and $\boldsymbol\Xi$}
The values associated with allocations $\boldsymbol\xi_i$ are denoted by regular fonts like $K_U$, $S_U$ or $V_U$ whereas the values associated with set of allocations $\boldsymbol{\Xi}$ are denoted by the Blackboard Bold fonts like $\mathbb{K}_U$, $\mathbb{S}_U$ or $\mathbb{V}_U$.
\subsubsection{Optimal visitation order and composite path}
Each allocation $\boldsymbol\xi_i$ has an optimal visitation order and optimal composite path denoted by $\boldsymbol{\lambda}^{\boldsymbol{\xi}_i}_i$ and $r^{\boldsymbol{\xi}_i}_{i,j}$ respectively. These satisfy the single-agent optimality condition \eqref{objfuncllnew} and have been obtained through the lower level graph search. They are defined in equation \eqref{lambdaxi=lambdaphi}.

\subsubsection{Unvisited interest points}
Each allocation $\boldsymbol\xi_i$ has number of unvisited interest points denoted by $K_U(\boldsymbol{\xi}_i)$.
\begin{equation} \label{unvisintptul}
K_U(\boldsymbol{\xi}_i) = K_L(\boldsymbol{\lambda}^{\boldsymbol{\xi}_i}_i,\boldsymbol{\xi}_i)
\end{equation}
Similarly, each set of allocation $\boldsymbol\Xi = \{\boldsymbol{\xi}_i\}$ has number of unvisited interest points, which is simply the sum of unvisited points over all its allocations,
\begin{equation}
\mathbb{K}_U(\boldsymbol\Xi) = \sum_{i=1}^{n_g} K_U(\boldsymbol\xi_i)
\end{equation}

\subsubsection{Arclength}
Each allocation $\boldsymbol\xi_i$ has an arclength $S_U(\boldsymbol{\xi}_i)$ which is the arclength of its optimal composite path.
\begin{equation} \label{arclengthul}
S_U(\boldsymbol{\xi}_i) = S_L(\boldsymbol{\lambda}^{\boldsymbol{\xi}_i}_i)
\end{equation}
The arclength for a set of allocations $\boldsymbol\Xi$ is defined as,
\begin{equation}
\mathbb{S}_U(\boldsymbol\Xi) = \sum_{i=1}^{n_g} S_U(\boldsymbol\xi_i)
\end{equation}

\subsubsection{Cost}
The cost of a set of allocations $\boldsymbol\Xi$ is defined similar to the lower level graph as,
\begin{equation}
\mathbb{V}_U(\boldsymbol\Xi) = \mathbb{S}_U(\boldsymbol{\Xi}) + P_U \mathbb{K}_U(\boldsymbol{\Xi})
\end{equation}
\begin{equation*}
 P_U =  \left( \sum_{i = 1}^{n_g} {h^0_{i}} + \sum_{j=1}^{n_t} h_{t_j} + 1 \right) / \tan(\gamma_{d_{min}})
\end{equation*}
The $P_U \mathbb{K}_U(\boldsymbol{\Xi})$ term, is a penalty term which increases the cost of allocations with more unvisited interest points and hence makes them less optimal. This is similar to the penalty term $P_L$ in the lower level graph search.

\subsubsection{Ideal cost}
The ideal cost for an allocation $\boldsymbol{\xi}_i$ is defined as,
\begin{equation} \label{idealcostul}
V^*_U(\boldsymbol{\xi}_i) =  V^*_L(\boldsymbol{\lambda}^*_i, \boldsymbol{\xi}_i)
\end{equation}
here $\boldsymbol{\lambda}^*_i$ is the node which satisfies the ideal optimality condition \eqref{objfuncllideal} for allocation $\boldsymbol{\xi}_i$ to glider $i$. This cost is actually not determinable since the node $\boldsymbol{\lambda}^*_i$ cannot be determined.

\subsubsection{Weak cost}
We also define another term called the weak cost for an allocation $\boldsymbol{\xi}_i$. 
\begin{equation} \label{lowerlimitul}
\underbar{V}_U(\boldsymbol{\xi}_i) =  \underbar{V}_L(\boldsymbol{\underbar{$\lambda$}}_i, \boldsymbol{\xi}_i)
\end{equation}
here $\boldsymbol{\underbar{$\lambda$}}_i$ is the node which satisfies the weak optimality condition \eqref{objfuncllweak} for allocation $\boldsymbol{\xi}_i$ to glider $i$. This cost is used the Branch\&Bound algorithm. The set of allocations $\boldsymbol\Xi$ also has a weak cost associated with it, defined as,
\begin{equation}
\underline{\mathbb{V}}_U(\boldsymbol\Xi) = \sum_{\forall i} \underbar{V}_U(\boldsymbol\xi_i) 
\end{equation}

\subsection{Upper level graph $\Gamma_U$}
Upper level graph denoted by $\Gamma_U$ is the graph, in which waypoints are allocated to different gliders.
A node in $\Gamma_{U}$ is a set of allocations $\boldsymbol{\Xi} = \{\boldsymbol{\xi}_i\}$. Each  $\boldsymbol{\xi}_i$ is the set of interest points allocated to a glider $i$.

The root (starting) node of $\Gamma_U$ is denoted by $\boldsymbol{\Xi}^0$. The waypoint allocations associated with this node, $\boldsymbol{\xi}^0_i = \{\}$ are empty sets, meaning that no waypoint has been allocated to any glider yet. The child $\boldsymbol{\Xi}'$ to a node $\boldsymbol{\Xi}$ is obtained by allocating an unallocated waypoint to one of the gliders. This is shown below,
\begin{multline}
	\boldsymbol{\Xi}' \in Children(\boldsymbol{\Xi}) \hspace{0.1cm} \text{only if} \hspace{0.1cm}  \exists l \in \{1,..,n_g\}\ \hspace{0.1cm} \text{s.t.} \\
	\boldsymbol{\xi}'_l = \left\{\begin{array}{ll}
		\boldsymbol{\xi}_l \cup \{\zeta\} \hspace{0.1cm} s.t. \hspace{0.1cm} \zeta \in \{1,..,n_w\} \setminus \bigcup_{i=1}^{n_g} \boldsymbol{\xi}_i \\ \boldsymbol{\xi}_j \hspace{0.1cm} \text{for} \hspace{0.1cm} \forall j \neq l \,.
	\end{array} \right.
\end{multline}
This means that all waypoint allocations for the child node $\boldsymbol{\xi}'_i$, are the same as the waypoint allocations for the parent node $\boldsymbol{\xi}_i$, except for $\boldsymbol{\xi}'_l$. The graph search on $\Gamma_U$ terminates when it encounters a goal node. The set of goal nodes of $\Gamma_U$ is denoted as $\Phi^f_U$. A node $\boldsymbol{\Xi} \in \Phi^f_U$, if $\bigcup_{i=1}^{n_g} \boldsymbol{\xi}_i = \{1,..,n_w\}$, meaning that it is exhaustive.

\subsection{Optimality condition}
The optimality condition for the multi-agent problem is to find the set of allocations which are exhaustive and minimize the total number of unvisited interest points. Moreover, if there are more than one such sets of allocations than we should find the one with the least accumulated arclength. This condition can be expressed in the current notation as,
\begin{equation} \label{optmulti}
\argmin_{\boldsymbol{\Xi}^*} \mathbb{S}_U(\boldsymbol{\Xi}^*) \hspace{0.1cm} \text{s.t.} \hspace{0.1cm} \boldsymbol{\Xi}^* \in \{\boldsymbol{\Xi} | \argmin_{\boldsymbol{\Xi}} \mathbb{K}_U(\boldsymbol{\Xi}) \} \cap \Phi^f_U
\end{equation}

Now consider another condition of optimality where we find the set of allocations which are exhaustive and minimize the cost $\mathbb{V}_U$.
\begin{equation} \label{optmultinewgraph}
\argmin_{\boldsymbol{\Xi}^*}\mathbb{V}_U({\boldsymbol{\Xi}^*}) \hspace{0.1cm} \text{s.t.} \hspace{0.1cm} \boldsymbol{\Xi}^* \in \Phi^f_U \,.
\end{equation}

Theorem \ref{lem3} guarantees that if a set of allocations $\boldsymbol\Xi $ satisfy this condition then it will satisfy the optimality condition for the multi-agent problem. For this to be true, $P_U$ has to be an upper bound on the distance the gliders can collectively travel. This can be expressed as an upper bound on collective height of the gliders divided by $\tan(\gamma_{d_{min}})$. An upper bound on height of the gliders is given by $\sum_{i = 1}^{n_g} {h^0_{i}} + \sum_{j=1}^{n_t} h_{t_j}$.
\begin{theorem} \label{lem3}                                                                                                              
If a node $\boldsymbol{\Xi}$ satisfies \eqref{optmultinewgraph} then it will satisfy the multi-agent optimality condition \eqref{optmulti}.
\end{theorem}
\begin{proof}
We prove this theorem in two parts. Let us assume $\hat{\boldsymbol{\Xi}}$ is the node which satisfies the condition \eqref{optmultinewgraph}.

First, we will prove that $\hat{\boldsymbol{\Xi}}$ has the least number of unvisited interest points in the set $\Phi^f_U$. This is a proof by contradiction. Suppose $\exists \tilde{\boldsymbol{\Xi}} \in \Phi^f_U$ and has fewer unvisited interest points, $\mathbb{K}_U(\tilde{\boldsymbol{\Xi}}) < \mathbb{K}_U(\hat{\boldsymbol{\Xi}})$. We have,
\begin{equation*}
\mathbb{V}_U(\hat{\boldsymbol{\Xi}}) \leq \mathbb{V}_U(\tilde{\boldsymbol{\Xi}})\,, \hspace{0.1cm} \text{(by definition)}
\end{equation*}
\begin{equation*}
\mathbb{S}_U(\hat{\boldsymbol{\Xi}}) + P_U \mathbb{K}_U(\hat{\boldsymbol{\Xi}}) \leq \mathbb{S}_U(\tilde{\boldsymbol{\Xi}}) + P_U \mathbb{K}_U(\tilde{\boldsymbol{\Xi}}) \,,
\end{equation*}
\begin{equation*}
\left(  \mathbb{K}_U(\hat{\boldsymbol{\Xi}}) - \mathbb{K}_U(\tilde{\boldsymbol{\Xi}}) \right) P_U \leq \mathbb{S}_U(\tilde{\boldsymbol{\Xi}}) - \mathbb{S}_U(\hat{\boldsymbol{\Xi}}) \,,
\end{equation*}
\begin{equation*}
\left( \mathbb{K}_U(\hat{\boldsymbol{\Xi}}) - \mathbb{K}_U(\tilde{\boldsymbol{\Xi}}) \right) P_U \leq \frac{ \sum_{i = 1}^{n_g} {h^0_{i}} + \sum_{j=1}^{n_t} h_{t_j} }{ \tan(\gamma_{d_{min}})} \,.
\end{equation*}
The last step was made possible by the fact that the arclength $\mathbb{S}_U$ is always upper bounded by the term on the right hand side. This term is the upper bound on the arclength. Using the definition of $P_U$ and the fact that $\mathbb{K}_U(\hat{\boldsymbol{\Xi}}) - \mathbb{K}_U(\tilde{\boldsymbol{\Xi}}) \geq 1$, we get a contradiction.

Secondly, we prove that $\hat{\boldsymbol{\Xi}}$ has the least accumulated arclength for all exhaustive sets of allocations which have the same number of unvisited interest points as $\hat{\boldsymbol{\Xi}}$. Assume  $\tilde{\boldsymbol{\Xi}}$ is such a set, this means it is exhaustive and $ \mathbb{K}_U(\hat{\boldsymbol{\Xi}}) = \mathbb{K}_U(\tilde{\boldsymbol{\Xi}})$. We have,
\begin{equation*}
\mathbb{V}_U(\hat{\boldsymbol{\Xi}}) \leq \mathbb{V}_U(\tilde{\boldsymbol{\Xi}})\,,
\end{equation*}
\begin{equation*}
\mathbb{S}_U(\hat{\boldsymbol{\Xi}}) + P_U \mathbb{K}_U(\hat{\boldsymbol{\Xi}}) \leq \mathbb{S}_U(\tilde{\boldsymbol{\Xi}}) + P_U \mathbb{K}_U(\tilde{\boldsymbol{\Xi}}) \,,
\end{equation*}
\begin{equation*}
\mathbb{S}_U(\hat{\boldsymbol{\Xi}}) \leq \mathbb{S}_U(\tilde{\boldsymbol{\Xi}}) \,.
\end{equation*}
Hence $\hat{\boldsymbol{\Xi}}$ has the minimum accumulated arclength for equal number of unvisited interest points.
\end{proof}

\subsection{Brute force search}
In our earlier work \cite{uzpath} we find a set of allocations which satisfy the condition \eqref{optmultinewgraph} by calculating the cost for each exhaustive set of allocations and picking the one that minimizes the cost. This is called brute force search and it is very computationally expensive. Hence in the next section we outline a Branch\&Bound type search, called the upper level graph search. This is also optimal and is shown to be faster than the brute force search via simulations.

%
\subsection{Upper level graph search}
Algorithm \ref{alg1} is a Branch\&Bound type graph search over $\Gamma_U$. The reason we cannot perform a uniform cost search on this graph is that we cannot prove monotonicity of cost. Monotonicity means that the cost of the child of a node $\boldsymbol{\Xi}$ will be strictly higher than the cost of the node itself. Instead we prove that the weak cost of a node $\underline{\mathbb{V}}_U(\boldsymbol\Xi) $ is actually the lower bound on the cost of its descendant. This is rigorously proven in theorem \ref{upperlevelbound}. As shown in \cite{bertsekas1995dynamic} this condition is sufficient to guarantee that Algorithm \ref{alg1} will returns the goal node which satisfies the optimality condition \eqref{optmultinewgraph}.

\begin{algorithm}[h]
	\caption {Upper Level Branch\&Bound Search}
	\label{alg1}
	\begin{algorithmic}[1]
		\STATE $\Psi = \{\boldsymbol{\Xi}^0\}$
		\STATE $\mathbb{V}_U(\boldsymbol{\Xi}^0)=0, Upper = \infty$
		\STATE $\mathbb{V}_U(\boldsymbol{\Xi})=\infty \quad \forall \boldsymbol{\Xi} \in {Nodes}(\Gamma_U) \textrm{ and } \boldsymbol{\Xi} \neq \boldsymbol{\Xi}^0$
		\WHILE {$\Psi \neq \{\}$}
		\STATE $\textrm{remove } \boldsymbol{\Xi} \textrm{ from } \Psi \textrm{ with the smallest } \mathbb{V}_U(\boldsymbol{\Xi})$
		\IF {$\underline{\mathbb{V}}_U({\boldsymbol{\Xi}}) < Upper$}
		\FORALL {$\boldsymbol{\Xi}' \in Children(\boldsymbol{\Xi}) \textbf{ and } \mathbb{V}_U(\boldsymbol{\Xi}') = \infty $}
		\STATE Calculate $\mathbb{V}_U(\boldsymbol{\Xi}')$ and $\underline{\mathbb{V}}_U({\boldsymbol{\Xi}'}) $ 
		\IF {$\boldsymbol{\Xi}' \in \Phi^f_U \textbf{ and }  \mathbb{V}_U(\boldsymbol{\Xi}') < Upper$}
		\STATE $Upper = \mathbb{V}_U(\boldsymbol{\Xi}')$
		\STATE $\hat{\boldsymbol{\Xi}} = \boldsymbol{\Xi}' $
		\ENDIF 
		\IF {$\underline{\mathbb{V}}_U({\boldsymbol{\Xi}'}) \leq Upper$}
		\STATE $ \text{Insert } \boldsymbol{\Xi}' \text{ in } \Psi \text{ with cost } \mathbb{V}_U(\boldsymbol{\Xi}')$
		\ENDIF
		\ENDFOR
		\ENDIF
		\ENDWHILE
		\RETURN $\hat{\boldsymbol{\Xi}}$
	\end{algorithmic}
\end{algorithm}
In the pseudo code for Algorithm \ref{alg1}, $\Psi$ is the Open Set (the set of nodes not yet expanded), and $Nodes(\Gamma_U)$ refers to the set of all possible nodes in $\Gamma_U$.

First some intermediate results are shown in lemmas \ref{V*U<V*U} and \ref{V*U<<V*U}. Then theorem \ref{upperlevelbound} is presents the main guarantee of the paper. 
\begin{lemma} \label{V*U<V*U}
Let $\boldsymbol{\Xi}''$ be a child of the node $\boldsymbol{\Xi}'$. The ideal cost of $\boldsymbol{\Xi}'$ will be a lower bound for ideal cost of $\boldsymbol{\Xi}''$, namely $V^*_U (\boldsymbol{\xi}'_i) \leq V^*_U (\boldsymbol{\xi}''_i)$.
\end{lemma}
\begin{proof}
The proof is given in the Appendix.
\end{proof}

\begin{lemma} \label{V*U<<V*U}
Let $\boldsymbol{\Xi}''$ be a descendant (child, child of child...) of the node $\boldsymbol{\Xi}'$. The ideal cost of $\boldsymbol{\Xi}'$ will be a lower bound for ideal cost of $\boldsymbol{\Xi}''$, namely $V^*_U (\boldsymbol{\xi}'_i) \leq V^*_U (\boldsymbol{\xi}''_i)$.
\end{lemma}
\begin{proof}
The proof of this lemma can be obtained by using the lemma \ref{V*U<V*U} recursively.
\end{proof}

\begin{theorem} \label{upperlevelbound}
Let $\boldsymbol{\Xi}''$ be a descendant (child, child of child ...) of the node $\boldsymbol{\Xi}'$. Then the weak cost of $\boldsymbol{\Xi}'$ will be a lower bound on the cost of $\boldsymbol{\Xi}''$, namely $\underline{\mathbb{V}}_U({\boldsymbol{\Xi}'}) \leq \mathbb{V}_U({\boldsymbol{\Xi}''}) $.
\end{theorem}
\begin{proof}
The proof is given in the Appendix.
\end{proof}

\section{Simulation Results} \label{sec:simul}
In this section we present an example of cooperative autonomous soaring of 2 gliders with 4 interest points and 4 thermals, meaning $n_g = 2, n_{ip} = 4, n_t = 4$. The maximum curvature $\kappa_{max} = 0.045 m^{-1}$ and maximum sharpness $\sigma_{max} = 0.001 m^{-2}$. Using these values of $\kappa_{max}$ and $\sigma_{max}$,
\begin{equation*}
\theta_{lim} = 2.02 \hspace{0.1cm} rad, \hspace{0.1cm} R_T = 33.8 \hspace{0.1cm} m
\end{equation*}
This satisfies assumption \ref{asm:elementary}. The starting positions, orientations, heights and final positions of the gliders are,
$$ p^0_1 = (445,709)^T m, \, p^0_2 = (646,754)^T m, $$
$$ \theta^0_1 = -1.41 rad, \, \theta^0_2 = 1.13 rad, \, h^0_1 = 600 m, \, h^0_2 = 500 m, $$
$$ p^f_1 = (765, 186)^T m, p^f_2 = (795,489)^T m $$
The positions of interest points and thermals are,
$$ p_{ip_1} = (97,950)^T m, \, p_{ip_2} = (823,34)^T m $$
$$ p_{ip_3} = (694,438)^T m, \, p_{ip_4} = (317,381)^T m $$
$$ p_{t_1} = (743,706)^T m, \, p_{t_2} = (392,32)^T m $$
$$ p_{t_3} = (655,277)^T m, \, p_{t_4} = (171,46)^T m $$
The minimum distance between waypoints $l_{min} = 108.4 \hspace{0.1cm} m$ and hence assumption \ref{asm:l_min>2R_T} is satisfied. The height a glider can gain by visiting the thermals and interest points is,
$$
h_{t_j} = 200 m, 1\leq j \leq 4, \, h_{ip_j} = 0 m, 1\leq j \leq 4.
$$
The minimum rate of descent of the gliders is chosen as $\gamma_{d_{min}} = 0.349$ radians. 

The set of interest point allocations which minimizes $V_U$ is $\hat{\boldsymbol{\Xi}} = \{\hat{\boldsymbol{\xi}}_i\}$, where the interest points allocated to the first gliders are $\hat{\boldsymbol{\xi}}_1 = \{ip_2, ip_4\}$ and the interest points allocated to the second glider are $\hat{\boldsymbol{\xi}}_2 = \{ip_1, ip_3\}$. The optimal visitation orders for each of these allocations are $\boldsymbol{\lambda}^{\hat{\boldsymbol{\xi}}_1}_1 = \{ip_4, t_3, ip_2\}$ and $\boldsymbol{\lambda}^{\hat{\boldsymbol{\xi}}_2}_2 = \{t_1, ip_1, ip_3\}$, respectively. The optimal paths associated with $\hat{\boldsymbol{\xi}}_i$, $r^{\hat{\boldsymbol{\xi}}_1}_{1,j}$ and $r^{\hat{\boldsymbol{\xi}}_2}_{2,j}$ are shown in the figure \ref{fig:Simulation}.
\begin{figure}[h!] 
\centering
\includegraphics[width=0.33\textwidth]{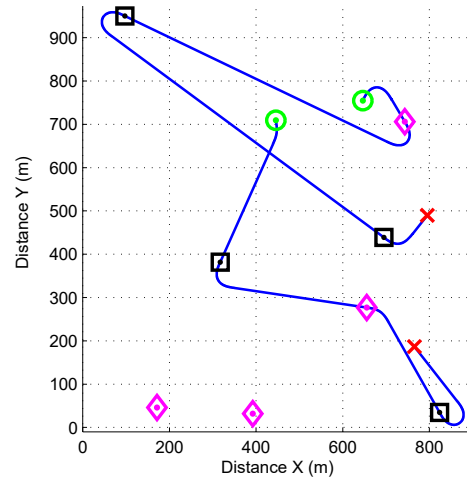}
\caption{Cooperative autonomous soaring for 2 gliders, 4 interest points and 4 thermals.}
\label{fig:Simulation}
\end{figure}

The curvature, sharpness and height profiles of the optimal paths are shown in figures \ref{fig:kappa} and \ref{fig:eta}. The paths are demonstrated to satisfy the endpoint and continuity constraints \eqref{boundarycond}-\eqref{contposorientcurv} and are also feasible and valid as they satisfy \eqref{cons:kappa_max}-\eqref{cons:h}.

Figure \ref{fig:intptchange} shows another scenario invloving 2 gliders, 3 interest points and a thermal. Here an interest point in glider 1's composite path is moved such that it comes closer to the composite path of glider 2. The algorithm determines that the more optimal solution where this interest point is in the composite path of glider 2.
\begin{figure}[h!] 
\centering
\includegraphics[width=0.5\textwidth]{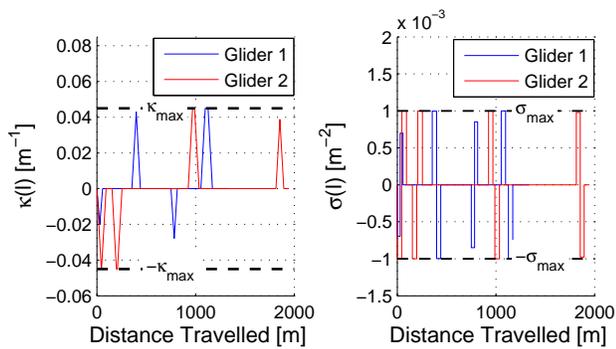}
\caption{Curvature and sharpness profiles.}
\label{fig:kappa}
\end{figure}
\begin{figure}[h!] 
\centering
\includegraphics[width=0.24\textwidth]{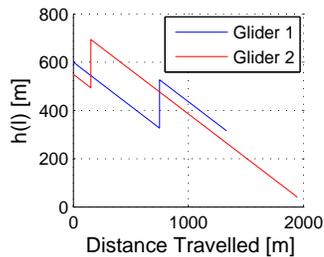}
\caption{Height profile.}
\label{fig:eta}
\end{figure}
\begin{figure}[t]
\centering
\begin{minipage}{.232\textwidth}
\centering
\includegraphics[width=1\textwidth]{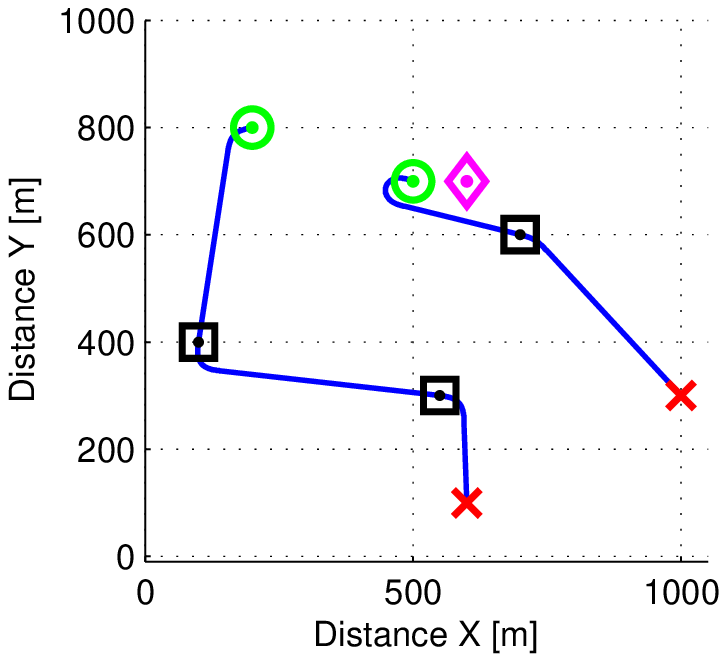}
\end{minipage} \hspace{0.1cm}
\begin{minipage}{.232\textwidth}
\centering
\includegraphics[width=1\textwidth]{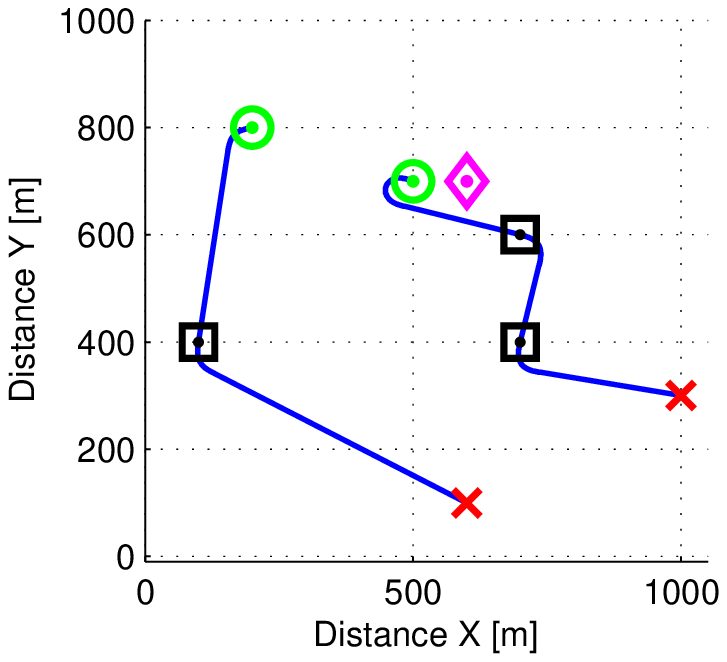}
\end{minipage}
\caption{How moving an interest point changes the optimal solution}
\label{fig:intptchange}
\end{figure}
\section{Conclusion} \label{sec:conc}
In this paper we solve the problem of planning paths for multiple gliders which have to visit a set interest points. The gliders can also use thermals to gain height and hence increase their range. First we introduce a special kind of curve, called a leg. Our algorithm uses this leg to generate composite paths for the gliders. The paths planned by the proposed algorithm are shown to be valid and feasible. Moreover they are guaranteed to cover as many interest points as possible within the particular class of composite paths introduced in the paper. Furthermore, they also have the minimum arclength of any other path with equal number of visited interest points.

Since, the approach proposed in this paper is exact (meaning it produces the best composite path form a class of composite paths) it is not particularly time efficient. Future work should focus on inexact but faster algorithms. Moreover, since the algorithm produces a composite path from a particular class of composite paths, it is sub-optimal. Hence, efforts should be made to provide bounds on measuring the sub-optimality of the algorithm.
\bibliographystyle{unsrt}
\bibliography{Bib}

\section*{Appendix}
\subsection{Proof of Lemma \ref{thm:epsilon1}}
The ratio is expressed as,
\begin{equation}
\frac{l^f}{l^e} = \frac{l^s - R_T \sin(\gamma) + l^{cc}}{l^e}
\end{equation}
First, the constraints on the variables $l^e$ and $\beta$ are calculated. Then, the upper bound of $l^f/l^e$ is calculated by obtaining the upper bound of another function $\mathcal{\bar R}$ which is always larger than $\frac{l^f}{l^e}$.

$l^e \geq l_{min}$ since $l_{min}$ is the minimum distance between any two waypoints. 
$\theta^e \in [0,\pi]$ as $\Delta y \geq 0$. To obtain the limits on $\beta$, equations \ref{theta_i} and \ref{eq:beta} are used. 
For $\theta^e = 0$, $\beta = 0$. 
For $\theta^e = \pi$, we get,
\begin{multline} \label{beta_{max}}
	\beta = \pi + 2 \arcsin\Big(\frac{R_M}{l^i}\Big) = \\
	\pi + 2 \arctan \Big(\frac{R_M}{l^e + R_T \sin(\gamma)}\Big)= \beta_{max}(l^e)
\end{multline}
hence, the $0 \leq \beta \leq \beta_{max}(l^e)$.

\begin{proof}
Consider the ratio $\mathcal{\bar R}$ which is always larger than or equal to $\frac{l^f}{l^e}$, where $\mathcal{\bar R}$ is a function of variables $l^e$ and $\beta$,
\begin{equation*}
	\mathcal{\bar{R}}(l^e, \beta) = \frac{\sqrt{(l^e + R_T)^2 - R_M^2} + R_T + l^{cc}(l^e, \beta)}{l^e} \geq \frac{l^f}{l^e} .
\end{equation*}
\begin{multline*}
\mathcal{\bar{R}}(l^e, \beta)
	= \mathcal{\bar{R}}_1 + \mathcal{\bar{R}}_2 \quad \text{for} \hspace{0.1 cm} \beta \in [0, \beta_{max}], l^e \in [l_{min},\infty),
\end{multline*}
where $\mathcal{\bar{R}}_1$ and $\mathcal{\bar{R}}_2$ are defined as,
\begin{multline*}
\mathcal{\bar{R}}_1(l^e) = \frac{\sqrt{(l^e + R_T)^2 - R_M^2}}{l^e}, \mathcal{\bar{R}}_2(l^e, \beta) = \frac{l^{cc} + R_T}{l^e} .
\end{multline*}
We upper bound the two functions seperately. An upper bound on $\mathcal{\bar{R}}_2$ can be obtained as shown below.
If $\beta \in [0, \theta_{lim})$,
\begin{multline} \label{mathcalR2}
\mathcal{\bar{R}}_2(l^e, \beta) = \frac{\beta/\kappa_{max} + \kappa_{max}/\sigma_{max} + R_T}{l^e}
\\ \leq \frac{\beta_{max}/\kappa_{max} + \kappa_{max}/\sigma_{max} + R_T}{l_{min}}
\end{multline}
If $\beta \in [0,\theta_{lim})$, $\mathcal{\bar{R}}_2$ is bounded in the following way. For this section we use the Fresnal integrals $C(\theta)$ and $S(\theta)$ as defined in the paper \cite{meek2004note} rather than $FrC(\theta)$ and $FrS(\theta)$ used in section \ref{planusingccturns}. $C(\theta)$ and $S(\theta)$ are defined as,
\begin{equation*}
C(\theta) = \int\limits_{0}^{\theta} \left( \cos(u)/\sqrt{u} \right) du \hspace{0.125cm} \text{and} \hspace{0.125cm} S(\theta) = \int\limits_{0}^{\theta} \left( \sin(u)/\sqrt{u} \right) du \,.
\end{equation*}
Through a change of variables it can be shown that,
\begin{equation*}
FrC(\theta)=C(\theta^2 \pi/2)/2 \pi \hspace{0.2cm} \text{and} \hspace{0.2cm} FrS(\theta)=S(\theta^2 \pi/2)/2 \pi \,.
\end{equation*}
Hence, the value of $\mathcal{\bar{R}}_2$ in case $\beta \in [0,\theta_{lim})$ becomes,
\begin{equation*}
\mathcal{\bar{R}}_2 = \frac{l^{cc} + R_T}{l^e} = \frac{2\sqrt{2 \beta} R_T \sin(\beta/2 - \gamma) + R_T}{l^e \left( \cos(\beta/2)C(\beta/2) + \sin(\beta/2)S(\beta / 2)\right)} \,.
\end{equation*}
We use the following inequalities from paper \cite{meek2004note},
\begin{equation*}
C(\theta) > 2 \sqrt{\theta} (1 - \theta^2 /10) \hspace{0.2cm} \text{and} \hspace{0.2cm} S(\theta) > 2 \theta^{3/2} (1 - \theta^2 /14)/3 \,,
\end{equation*}
and the fact that $\beta \leq \theta_{lim} \leq \pi$ and $l^e \geq l_{min} > 0$, to obtain,
\begin{equation*}
\frac{l^{cc}}{l^e} < \frac{2 R_T \sin(\beta/2 - \gamma) + R_T}{ l^e \left(\cos(\beta/2)(1 - \beta^2/40) + \beta \sin(\beta/2)(1 - \beta^2 /56)/6 \right)} \,,
\end{equation*}
\begin{equation*}
< (4.66 R_T \sin(\beta/2 - \gamma) + R_T)/l^e < 5.66 R_T/l_{min} \,.
\end{equation*}
Hence, $\mathcal{\bar{R}}_2 < \mathcal{R}^{cc}_{max}$. \\
The upper bound of the function $\mathcal{\bar R}_1$ is found, using the Karush-Kuhn-Tucker (KKT) conditions for finding a constrained optima. The derivative of $\mathcal{\bar{R}}_1(l^e)$ w.r.t. $l^e$ is,
\begin{equation*}
\frac{\partial \mathcal{\bar R}_1}{\partial l^e} = 
\frac{-1}{{l^e}^2} \bigg(\frac{R_T l^e + R_T^2\sin^2(\gamma)}{\sqrt{(l^e + R_T)^2 - R_M^2}} \bigg).
\end{equation*}
The problem now becomes,
\begin{equation}
	\min f = -\mathcal{\bar R}_1(l^e) = -\frac{\sqrt{(l^e + R_T)^2 - R_M^2}}{l^e}
\end{equation}
subject to the constraint,
\begin{equation}
	g_1(l^e,\beta) = l_{min} - l^e \leq 0 ,
\end{equation}
The KKT conditions guarantee that exist a variable $\mu_1$, called the KKT multiplier, which satisfies the following necessary conditions, \\
\begin{equation} \label{eq:opt1}
	\frac{\partial f}{\partial l^e} + \mu_1 \frac{\partial g_1}{\partial l^e} = \frac{1}{{l^e}^2} \bigg(\frac{R_T l^e + R_T^2\sin^2(\gamma)}{\sqrt{(l^e + R_T)^2 - R_M^2}} \bigg) - \mu_1 = 0
\end{equation}
\begin{equation} \label{eq:comp1}
	\mu_1 g_1(l^e,\beta) = \mu_1 (l_{min} - l^e) = 0
\end{equation}
\begin{equation} \label{eqmu_1geq0}
	\mu_1 \geq 0
\end{equation}
Let us assume $\mu_1 = 0$. We know $\frac{\partial f}{\partial l^e} > 0$ in equation \eqref{eq:opt1}, since numerator $R_T l^e + R_T^2\sin^2(\gamma) > 0$, due to $l^e \geq l_{min} > 0$. So, if $\mu_1 = 0$ then the optimality condition \eqref{eq:opt1} does not hold. Hence it is proved that $\mu_1 \neq 0$. This fact when used with equation \eqref{eq:comp1}, gives us $l^e = l_{min}$. So, $l^e = l_{min}$ is a stationary point of $f$, since it satisfies the necessary conditions for optimality. To prove that it is an minima it has to satisfy the sufficient conditions as well.
The KKT sufficient conditions for optimality for univariate functions like $f$ is that the double derivative of $f$ should be positive.
\begin{multline*}
\frac{\partial^2 f}{{\partial l^e}^2} = \\ \frac{2 R_T {l^e}^3 + 3R_T^2(1+\sin^2(\gamma)){l^e}^2 + 6 R_T^3 \sin^2(\gamma)l^e + 2R_T^4\sin^4(\gamma)}{{l^e}^3 ({l^e}^2 + 2 R_T l^e + R_T^2 \sin^2(\gamma))} \,,
\end{multline*}
$\frac{\partial^2 f}{{\partial l^e}^2} > 0$ since all the terms in the numerator and denominator are positive due to the facts $l^e \geq l_{min} > 0$ and $R_T>0$. So this satisfies the KKT sufficient condition for global optimality and hence $\mathcal{\bar{R}}_1(l_{min})$ is the maximum value of the function $\mathcal{\bar{R}}_1(l^e)$. In conclusion,
\begin{multline*}
\mathcal{\bar{R}}(l^e, \beta) = \mathcal{\bar{R}}_1(l^e) + \mathcal{\bar{R}}_2(l^e, \beta) \leq \mathcal{\bar{R}}_1(l_{min}) + \mathcal{\bar{R}}_2(l_{min}, \beta_{max}) \\ = \frac{\sqrt{(l_{min} + R_T)^2 - R_M^2} + R_T + \frac{\beta_{max}(l_{min})}{\kappa_{max}} + \frac{\kappa_{max}}{\sigma_{max}}}{l_{min}} = \mathcal{R}_{max}
\end{multline*}

\end{proof}

\subsection{Proof of Lemma \ref{V*U<V*U}.}
Let $j \in \{1,..,n_g\}$ be the allocation which has changed from $\boldsymbol{\Xi}'$ to $\boldsymbol{\Xi}''$ and let $\hat\xi$ be the new interest point that has been allocated to $j$. Hence $ \boldsymbol{\xi}''_j = \boldsymbol{\xi}'_j \cup \{\hat{\xi}\}$ and $\forall i \neq j, \hspace{0.1cm} \boldsymbol{\xi}'_i = \boldsymbol{\xi}''_i$. Therefore, $\forall i \neq j$, $V^*_U (\boldsymbol{\xi}'_i) \leq V^*_U (\boldsymbol{\xi}''_i)$ is trivially satisfied.

For this proof the nodes which satisfy ideal optimality condition \eqref{objfuncllideal} for allocations $\boldsymbol{\xi}''_j$ and $\boldsymbol{\xi}'_j$ are denoted by $\boldsymbol{\lambda}^{**}_j$ and $\boldsymbol{\lambda}^*_j$ respectively. From the definition \eqref{idealcostul}, we know that $V^*_U (\boldsymbol{\xi}'_j) =  V^*_L (\boldsymbol{\lambda}^*_j, \boldsymbol{\xi}'_j)$ and $V^*_U (\boldsymbol{\xi}''_j) = V^*_L (\boldsymbol{\lambda}^{**}_j, \boldsymbol{\xi}''_j)$.

Before we start the proof, lets define some important sets. Let $\boldsymbol{\Lambda}'_j$ as the set of all valid goal nodes in graph $\Gamma_L(\boldsymbol{\xi}'_j)$. Similarly we define $\boldsymbol{\Lambda}''_j$ denote the set of all valid goal nodes in graph $\Gamma_L(\boldsymbol{\xi}''_j)$. Since $\boldsymbol{\Lambda}'_j$ will contain all possible valid combinations of thermals and interest points from the set $\boldsymbol{\xi}'_j$ and $\boldsymbol{\Lambda}''_j$ will contain all possible valid combinations of thermals and interest points from the set $\boldsymbol{\xi}''_j$ and we also know that $\boldsymbol{\xi}'_j \subset \boldsymbol{\xi}''_j$, its quite easy to see that $\boldsymbol{\Lambda}'_j \subset \boldsymbol{\Lambda}''_j$. 

Similarly we define $\hat{\boldsymbol{\Lambda}}'_j$ as the set of all \emph{ideally valid} goal nodes in graph $\Gamma_L(\boldsymbol{\xi}'_j)$ and $\hat{\boldsymbol{\Lambda}}''_j$ as the set of all \emph{ideally valid} goal nodes in graph $\Gamma_L(\boldsymbol{\xi}''_j)$. By similar reasoning we can deduce $\hat{\boldsymbol{\Lambda}}'_j \subset \hat{\boldsymbol{\Lambda}}''_j$. From lemma \ref{lem:validsets} we can also see that $\boldsymbol{\Lambda}'_j \subset \hat{\boldsymbol{\Lambda}}'_j$ and $\boldsymbol{\Lambda}''_j \subset \hat{\boldsymbol{\Lambda}}''_j$.

The definition of $V^*_L ()$ as described in section \ref{sec:singleglider} is also dependent on the interest points allocation $\boldsymbol{\xi}_j$. This is due to the unvisited interest points $K_L (\boldsymbol{\lambda}_j,\boldsymbol{\xi}_j)$ in the definition of $V^*_L ()$. If we define $k_{ip}(\boldsymbol{\lambda}_j)$ as the number of interest points in $\boldsymbol{\lambda}_j$, then,
\begin{equation*}
K_L (\boldsymbol{\lambda}_j,\boldsymbol{\xi}_j) = n(\boldsymbol{\xi}_j) - k_{ip}(\boldsymbol{\lambda}_j)
\end{equation*}
where $n$ depends on just the allocation and $k_{ip}$ depends on just the visitation order.

Since $ \boldsymbol{\xi}''_j = \boldsymbol{\xi}'_j \cup \{\hat{\xi}\}$, we know that $n(\boldsymbol{\xi}''_j) = n(\boldsymbol{\xi}'_j) + 1$. Substituting in $K_L$ and $k_{ip}$ we get, 
\begin{equation} \label{KLkip=KLkip+1}
K_L(\boldsymbol{\lambda}^{**}_j, \boldsymbol{\xi}''_j) + k_{ip}(\boldsymbol{\lambda}^{**}_j) = K_L(\boldsymbol{\lambda}^*_j, \boldsymbol{\xi}'_j) + k_{ip}(\boldsymbol{\lambda}^*_j) + 1
\end{equation}

\underbar{1) Consider the case where $K_L(\boldsymbol{\lambda}^{**}_j, \boldsymbol{\xi}''_j) = K_L(\boldsymbol{\lambda}^*_j, \boldsymbol{\xi}'_j)$}. Using the above shown equation we get $k_{ip}(\boldsymbol{\lambda}^{**}_j) = k_{ip}(\boldsymbol{\lambda}^*_j) + 1$. This means that $\boldsymbol{\lambda}^{**}_j$ has one more interest point than $\boldsymbol{\lambda}^*_j$. We know that $\boldsymbol{\lambda}^*_j$ is the ideally valid goal node with the least ideal cost in the graph $\Gamma_L(\boldsymbol{\xi}'_j)$. From lemma \ref{lem4_ideal} we can deduce that the $\boldsymbol{\lambda}^*_j$ has the most interest points in the set $\hat{\boldsymbol{\Lambda}}'_j$. Hence, $\boldsymbol{\lambda}^{**}_j \notin \hat{\boldsymbol{\Lambda}}'_j$ because we know that $ k_{ip}(\boldsymbol{\lambda}^*_j) < k_{ip}(\boldsymbol{\lambda}^{**}_j)$. This also means $\boldsymbol{\lambda}^{**}_j \notin \boldsymbol{\Lambda}'_j$. But by definition $\boldsymbol{\lambda}^{**}_j \in \boldsymbol{\Lambda}''_j$. This means that $\boldsymbol{\lambda}^{**}_j$ contains the interest point $\hat{\xi}$.

Let us define another visitation order $\tilde{\boldsymbol{\lambda}}_j$ by removing $\hat{\xi}$ from $\boldsymbol{\lambda}^{**}_j$. This means that $k_{ip}(\tilde{\boldsymbol{\lambda}}_j) = k_{ip}(\boldsymbol{\lambda}^*_j)$. $\tilde{\boldsymbol{\lambda}}_j \in \boldsymbol{\Lambda}'_j$ since it only contains interest points in $\boldsymbol{\xi}'_j$. By removing an interest point from $\boldsymbol{\lambda}^{**}_j$ the arclength of its ideal path cannot increase, hence $S^*_L(\tilde{\boldsymbol{\lambda}}_j) \leq S^*_L(\boldsymbol{\lambda}^{**}_j)$. Using this piece of information with the definition of ideal validity we can deduce that $\tilde{\boldsymbol{\lambda}}_j$ will be ideally valid. Hence $\tilde{\boldsymbol{\lambda}}_j \in \hat{\boldsymbol{\Lambda}}'_j$. We also know from lemma \ref{lem4_ideal} that $\boldsymbol{\lambda}^*_j$ has the smallest ideal arclength of all visitation orders with interest points equal to $k_{ip}(\boldsymbol{\lambda}^*_j)$ in the set $\hat{\boldsymbol{\Lambda}}'_j$. And since we know that $k_{ip}(\tilde{\boldsymbol{\lambda}}_j) = k_{ip}(\boldsymbol{\lambda}^*_j)$ and $\tilde{\boldsymbol{\lambda}}_j \in \hat{\boldsymbol{\Lambda}}'_j$, 

\begin{equation*}
S^*_L (\boldsymbol{\lambda}^*_j) \leq S^*_L(\tilde{\boldsymbol{\lambda}}_j) \leq S^*_L (\boldsymbol{\lambda}^{**}_j) \,,
\end{equation*}
\begin{equation*}
S^*_L (\boldsymbol{\lambda}^*_j) + K_L(\boldsymbol{\lambda}^*_j, \boldsymbol{\xi}'_j)P_L \leq S^*_L (\boldsymbol{\lambda}^{**}_j) + K_L(\boldsymbol{\lambda}^{**}_j, \boldsymbol{\xi}''_j)P_L \,,
\end{equation*}
\begin{equation*}
V^*_L (\boldsymbol{\lambda}^*_j) \leq V^*_L (\boldsymbol{\lambda}^{**}_j) \Rightarrow V^*_U (\boldsymbol{\xi}'_j) \leq V^*_U (\boldsymbol{\xi}''_j) \,.
\end{equation*}

\underbar{2) Consider the case where $K_L(\boldsymbol{\lambda}^*_j, \boldsymbol{\xi}'_j) < K_L(\boldsymbol{\lambda}^{**}_j, \boldsymbol{\xi}''_j)$}. We know that $P_L > S^{max}_L(\boldsymbol{\lambda}_i)$ for any $\boldsymbol{\lambda}_i$. We also know that $\boldsymbol{\lambda}^*_j, \boldsymbol{\lambda}^{**}_j$ are ideally valid nodes which means that $S^*_L(\boldsymbol{\lambda}^*_j) \leq S^{max}_L(\boldsymbol{\lambda}^*_j)$and $S^*_L(\boldsymbol{\lambda}^{**}_j) \leq S^{max}_L(\boldsymbol{\lambda}^{**}_j)$. Hence we get $S^*_L(\boldsymbol{\lambda}^*_j), S^*_L(\boldsymbol{\lambda}^{**}_j) \leq P_L$. Since $K_L$ can only have integer values we know $K_L(\boldsymbol{\lambda}^{**}_j, \boldsymbol{\xi}''_j) - K_L(\boldsymbol{\lambda}^*_j, \boldsymbol{\xi}'_j) \geq 1$.
\begin{equation*}
1 \leq K_L(\boldsymbol{\lambda}^{**}_j, \boldsymbol{\xi}''_j) - K_L(\boldsymbol{\lambda}^*_j, \boldsymbol{\xi}'_j)
\end{equation*}
\begin{equation*}
S^*_L(\boldsymbol{\lambda}^*_j) \leq P_L \leq K_L(\boldsymbol{\lambda}^{**}_j, \boldsymbol{\xi}''_j) P_L - K_L(\boldsymbol{\lambda}^*_j, \boldsymbol{\xi}'_j) P_L 
\end{equation*}
\begin{equation*}
S^*_L (\boldsymbol{\lambda}^*_j) + K_L(\boldsymbol{\lambda}^*_j, \boldsymbol{\xi}'_j)P_L \leq K_L(\boldsymbol{\lambda}^{**}_j, \boldsymbol{\xi}''_j)P_L
\end{equation*}
\begin{equation*}
S^*_L (\boldsymbol{\lambda}^*_j) + K_L(\boldsymbol{\lambda}^*_j, \boldsymbol{\xi}'_j)P_L \leq S^*_L (\boldsymbol{\lambda}^{**}_j) + K_L(\boldsymbol{\lambda}^{**}_j, \boldsymbol{\xi}''_j)P_L
\end{equation*}
\begin{equation*}
V^*_L (\boldsymbol{\lambda}^*_j) \leq V^*_L (\boldsymbol{\lambda}^{**}_j) \Rightarrow V^*_U (\boldsymbol{\xi}'_j) \leq V^*_U (\boldsymbol{\xi}''_j) \,.
\end{equation*}

\underbar{3) Finally we assume that $K_L(\boldsymbol{\lambda}^*_j, \boldsymbol{\xi}'_j) > K_L(\boldsymbol{\lambda}^{**}_j, \boldsymbol{\xi}''_j)$}. We show that this is impossible by using contradiction. Using the relation \eqref{KLkip=KLkip+1} along with the above shown expression we get,
\begin{equation*}
k_{ip}(\boldsymbol{\lambda}^{**}_j) > k_{ip}(\boldsymbol{\lambda}^*_j) + 1
\end{equation*}
We know that $\boldsymbol{\lambda}^{**}_j \in \hat{\boldsymbol{\Lambda}}''_j$. But we can say for sure that $\boldsymbol{\lambda}^{**}_j \notin \hat{\boldsymbol{\Lambda}}'_j$ since $\boldsymbol{\lambda}^*_j$ is known to have the highest number of interest points in $\hat{\boldsymbol{\Lambda}}'_j$ but $\boldsymbol{\lambda}^{**}_j$ has at least two more interest points than $\boldsymbol{\lambda}^*_j$. This means that $\boldsymbol{\lambda}^{**}_j$ contains the new interest point $\hat\xi$. Lets define another visitation order $\tilde{\boldsymbol{\lambda}}_j$ by removing $\hat\xi$ from $\boldsymbol{\lambda}^{**}_j$. Now we get $k_{ip}(\tilde{\boldsymbol{\lambda}}_j) > k_{ip}(\boldsymbol{\lambda}^*_j)$. We know that $\tilde{\boldsymbol{\lambda}}_j$ will have an ideally valid ideal path and it will only contain interest points from the set $\boldsymbol{\xi}'_j$, hence it must be in the set $\hat{\boldsymbol{\Lambda}}'_j$. 
This leads us to a contradiction since $k_{ip}(\tilde{\boldsymbol{\lambda}}_j) > k_{ip}(\boldsymbol{\lambda}^*_j)$ but $\boldsymbol{\lambda}^*_j$ has the most interest points in the set $\hat{\boldsymbol{\Lambda}}'_j$. This proves that the assumption $K_L(\boldsymbol{\lambda}^*_j, \boldsymbol{\xi}'_j) > K_L(\boldsymbol{\lambda}^{**}_j, \boldsymbol{\xi}''_j)$ was invalid.

\subsection{Proof of Theorem \ref{upperlevelbound}}
Lets assume $\boldsymbol{\Xi}''$ is a descendant (child, child of child ...) of the node $\boldsymbol{\Xi}'$. Let the $\mathbb{G} \subset \{1,..,n_g\}$ be the set of gliders whose allocations have changed from $\boldsymbol{\Xi}'$ to $\boldsymbol{\Xi}''$. Hence $\boldsymbol{\xi}'_i \subset \boldsymbol{\xi}''_i, \hspace{0.1cm} \forall i \in \mathbb{G} $ and $\boldsymbol{\xi}''_i = \boldsymbol{\xi}'_i, \hspace{0.1cm} \forall i \notin \mathbb{G} $.

We start by showing that 
\begin{equation} \label{upperlevelboundconjecture}
\underbar{V}_U(\boldsymbol{\xi}'_i) \leq S_U (\boldsymbol{\xi}''_i) + K_U(\boldsymbol{\xi}''_i)P_L
\end{equation}
for each glider $i$. In this proof the nodes which satisfy optimality condition \eqref{objfuncllnew} for allocations $\boldsymbol{\xi}'_i$ and $\boldsymbol{\xi}''_i$ are denoted by $\boldsymbol{\lambda}'_i$ and $\boldsymbol{\lambda}''_i$ respectively. Similarly the nodes which satisfy ideal optimality condition \eqref{objfuncllideal} for allocations $\boldsymbol{\xi}'_i$ and $\boldsymbol{\xi}''_i$ are denoted by $\boldsymbol{\lambda}^*_i$ and $\boldsymbol{\lambda}^{**}_i$ respectively. Furthermore the node which satisfies weak optimality condition \eqref{objfuncllweak} for allocation $\boldsymbol{\xi}'_i$ is denoted by $\underbar{$\boldsymbol{\lambda'}$}_i$.

\underbar{1) In case when $i \notin \mathbb{G}$}, conjecture \eqref{upperlevelboundconjecture} is easy to prove since, $\boldsymbol{\xi}''_i = \boldsymbol{\xi}'_i$ which means $\boldsymbol{\lambda}''_i = \boldsymbol{\lambda}'_i$ and $\boldsymbol{\lambda}^{**}_i = \boldsymbol{\lambda}^*_i$
\begin{equation*}
\underbar{V}_U(\boldsymbol{\xi}'_i) = \underbar{V}_L(\underbar{$\boldsymbol{\lambda'}$}_i, \boldsymbol{\xi}'_i) \hspace{0.1cm} \text{(definition \eqref{lowerlimitul})}
\end{equation*}
\begin{equation*}
\underbar{V}_U(\boldsymbol{\xi}'_i) \leq \underbar{V}_L(\boldsymbol{\lambda}'_i, \boldsymbol{\xi}'_i) \hspace{0.1cm} \text{(definition \eqref{objfuncllweak} and lemma \ref{lem:validsets})}
\end{equation*}
\begin{equation*}
\underbar{V}_U(\boldsymbol{\xi}'_i) \leq V_L(\boldsymbol{\lambda}'_i, \boldsymbol{\xi}'_i) \hspace{0.1cm} \text{(definition \eqref{eq:defweakcost})}
\end{equation*} 
\begin{equation*}
\underbar{V}_U(\boldsymbol{\xi}'_i) \leq V_L(\boldsymbol{\lambda}'_i, \boldsymbol{\xi}'_i) = V_L(\boldsymbol{\lambda}''_i, \boldsymbol{\xi}''_i) \hspace{0.1cm} \text{since} \hspace{0.1cm} \boldsymbol{\xi}''_i = \boldsymbol{\xi}'_i, \hspace{0.1cm} \boldsymbol{\lambda}''_i = \boldsymbol{\lambda}'_i
\end{equation*}
\begin{multline*}
\underbar{V}_U(\boldsymbol{\xi}'_i) \leq V_L(\boldsymbol{\lambda}''_i, \boldsymbol{\xi}''_i) = S_L(\boldsymbol{\lambda}''_i) + K_L(\boldsymbol{\lambda}''_i, \boldsymbol{\xi}''_i)P_L \\
= S_U(\boldsymbol{\xi}''_i) + K_U(\boldsymbol{\xi}''_i)P_L \hspace{0.1cm} \text{(definitions \eqref{unvisintptul},\eqref{arclengthul})}
\end{multline*}

\underbar{2) In case when $i \in \mathbb{G}$}, the proof of conjecture \eqref{upperlevelboundconjecture} can be broken down into two parts.

a) Firstly, we show that, $\underbar{V}_U(\boldsymbol{\xi}'_i) \leq V^*_U (\boldsymbol{\xi}'_i)$. 

Using definition \eqref{lowerlimitul} we get,
\begin{equation*}
\underbar{V}_U(\boldsymbol{\xi}'_i) = \underbar{V}_L(\underbar{$\boldsymbol{\lambda'}$}_i, \boldsymbol{\xi}'_i)
\end{equation*}

Using lemma \ref{lem:Vund<V*} we get,
\begin{equation*}
\underbar{V}_U(\boldsymbol{\xi}'_i) = \underbar{V}_L(\underbar{$\boldsymbol{\lambda'}$}_i, \boldsymbol{\xi}'_i) \leq V^*_L(\boldsymbol{\lambda}^*_i, \boldsymbol{\xi}'_i)
\end{equation*}

Using definition \eqref{idealcostul} we get,
\begin{equation} \label{VU<V*U}
\underbar{V}_U(\boldsymbol{\xi}'_i) \leq V^*_U (\boldsymbol{\xi}'_i)
\end{equation}

b) Then we show that, $V^*_U (\boldsymbol{\xi}''_i) \leq S_U (\boldsymbol{\xi}''_i) + K_U(\boldsymbol{\xi}''_i)P_L $.

From definitions \eqref{unvisintptul}, \eqref{arclengthul} and \eqref{idealcostul} we know that $ V^*_U (\boldsymbol{\xi}''_i) = V^*_L (\boldsymbol{\lambda}^{**}_i, \boldsymbol{\xi}''_i)$ and $S_U (\boldsymbol{\xi}''_i) + K_U(\boldsymbol{\xi}''_i)P_L =  V_L(\boldsymbol{\lambda}''_i, \boldsymbol{\xi}''_i)$.

From lemma \ref{lem:V*<V} we know that $V^*_L (\boldsymbol{\lambda}^{**}_i, \boldsymbol{\xi}''_i) \leq V_L(\boldsymbol{\lambda}''_i, \boldsymbol{\xi}''_i)$, hence we get,
\begin{equation} \label{V*U<SUKU}
V^*_U (\boldsymbol{\xi}''_i) \leq S_U (\boldsymbol{\xi}''_i) + K_U(\boldsymbol{\xi}''_i)P_L
\end{equation}

If we use equations \eqref{VU<V*U} and \eqref{V*U<SUKU} with lemma \ref{V*U<<V*U}, we get,
\begin{equation*}
\underbar{V}_U(\boldsymbol{\xi}'_i) \leq S_U (\boldsymbol{\xi}''_i) + K_U(\boldsymbol{\xi}''_i)P_L
\end{equation*}

Now if we sum over all gliders we come up with,
\begin{equation*}
\underline{\mathbb{V}}_U(\boldsymbol{\Xi}') \leq \sum_{\forall i} S_U (\boldsymbol{\xi}''_i) + \sum_{\forall i} K_U(\boldsymbol{\xi}''_i)P_L
\end{equation*}
\begin{equation*}
\underline{\mathbb{V}}_U(\boldsymbol{\Xi}') \leq \mathbb{V}_U({\boldsymbol{\Xi}''}), \hspace{0.1cm} \text{(since $P_L < P_U$ )}
\end{equation*}
\end{document}